\documentclass[12pt,a4paper]{amsart}

\usepackage{amsmath,amsthm,amsfonts,amssymb,array,amscd}
\usepackage{amsmath,geometry,amssymb,amsfonts,amsthm,graphicx,enumerate,latexsym,tabularx,amscd}

\theoremstyle{plain}
\usepackage{amsthm}

\newtheorem{theorem}{Theorem}[subsection]
\newtheorem{definition}{Definition}[subsection]

\newtheorem{lemma}{Lemma}[subsection]
\usepackage[OT2,T1]{fontenc}
\DeclareSymbolFont{cyrletters}{OT2}{wncyr}{m}{n}
\DeclareMathSymbol{\Sha}{\mathalpha}{cyrletters}{"58}
\newtheorem{claim}{Claim}[subsection]
\newtheorem{proposition}{Proposition}[subsection]
\newtheorem{corollary}{Corollary}[subsection]
\newtheorem{question}{Question}[subsection]

\newtheorem{conjecture}{Conjecture}

\theoremstyle{definition}
 \newtheorem{rem}[theorem]{Remark}
 \numberwithin{equation}{section}
\theoremstyle{definition}
\theoremstyle{remark}
 \numberwithin{equation}{section}

\renewcommand{\leq}{\leqslant}
\renewcommand{\geq}{\geqslant}
\renewcommand{\setminus}{\smallsetminus}
\usepackage[pdftex]{hyperref}
\usepackage{graphicx}
\usepackage{amsthm,enumitem} 
\usepackage{enumitem}
\setlist[enumerate]{itemsep=0mm}
\usepackage{tikz-cd, enumerate}
\graphicspath{ {Figures/} }

\renewcommand{\and}{\quad \mbox{and} \quad}  
\renewcommand{\leq}{\leqslant}
\renewcommand{\geq}{\geqslant}
\renewcommand{\setminus}{\smallsetminus}
\title{\bf Virtual characters on L-functions}

\author[Subham Bhakta]{\bfseries Subham Bhakta}

\address{
Chennai Mathematical Institute\\ 
H1, Sipcot It Park, Siruseri  \\ 
Kelambakkam, 603103\\
India}
\email{Subham@cmi.ac.in, Subham1729@gmail.com}

\begin{document}
\vspace{10mm}
\setcounter{page}{1}
\thispagestyle{empty}

\maketitle

\begin{abstract} In this expository note we show the inception and development of the Heilbronn characters and their application to the holomorphy of quotients of Artin L-functions. Further we use arithmetic Heilbronn characters introduced by Wong, to deal with holomorphy of quotients of certain L-functions, e.g,, L-functions associated to CM elliptic curves. Furthermore we use the supercharacter theory introduced by Diaconis and Isaacs to study Artin L-functions associated to such characters. We conclude the note surveying about various other unconditional approaches taken based on character theory of finite groups.
\end{abstract}
\section{Introduction}
A natural number $p$ is said to be prime if it has exactly two divisors, namely, $1$ and
itself. More than two thousand years ago that Euclid proved the infinitude of primes in a elementary way. For a sake of contradiction, Euclid assumed there are only finitely many primes $p_1,p_2,\cdots, p_k$ and got contradiction considering the number $(\prod_{i=1}^{k}p_i)+1$. In 1737, Euler gave another proof via an analytic method. He considered the sum $f(x)=\sum \frac{1}{n^x}$ for $x>1$ and showed the summation diverges as $x\to 1^{+}$. From this Euler concluded that $\sum \frac{1}{p}$ diverges, where the summation runs over all primes $p$. The function $f(x)$ gave birth of a new function (on complex plane),
\[\zeta(s)=\sum\frac{1}{n^s},\] 
in $\mbox{Re}(s)>1$. The function is widely known as Riemann-zeta function. This led Dedekind to his zeta
functions associated to a number field $K$ as, 
\[\zeta_K(s)=\sum \frac{1}{N(\mathfrak{a})^s},\]
in $\mbox{Re}(s)>1$, where the summation runs over all non trivial ideals $\mathfrak{a}$ of $O_K$ and $N(\mathfrak{a})$ is index of the ideal $\mathfrak{a}$ in $O_K$. Where $O_K$ is the ring of integers in $K$. This generalizes zeta function because, for $K=\mathbb{Q}, \zeta_K(s)=\zeta(s)$. A celebrated conjecture of Dedekind asserts that for any finite algebraic extension $F$ of $\mathbb{Q}$, the zeta function $\zeta_F(s)$ is divisible by the Riemann zeta function $\zeta(s)$. That
is, the quotient $\zeta_F (s)/\zeta(s)$ is entire. More generally, Dedekind conjectured that if $K$ is a finite extension of $F$, then $\zeta_K(s)/\zeta_F (s)$ is entire. This conjecture is still open. By the work of Aramata and Brauer the conjecture is known if $K/F$ is Galois. We record this in Theorem 4.10. If $K$ is contained in a solvable extension of $F$, then Uchida and van der Waall independently proved Dedekind's conjecture, we discuss this in Theorem 3.1.1. \\
\newline 
In 1923, for every Galois extension $K/F$ of number fields with Galois group $G$, and every representation $\rho$ of G into $\mbox{GL}(V)$, ($V$ is a finite dimensional vector space over $\mathbb{C}$). Artin defined a L-function attached to $\rho$. Artin conjectured that his L-function can be extended to an entire function whenever $\rho$ is non-trivial and irreducible. Artin showed that his conjecture is true for one dimensional characters, based on his reciprocity law. From this and the induction-invariance property of Artin L-functions, Artin established his conjecture when $G$ is an $M$-group. An $M$-group a group, for which every irreducible characters are can be written as sum one induced (one dimensional) characters. Brauer proved that all Artin L-functions extends meromorphically over $\mathbb{C}$. One can show ratio of two Dedekind zeta functions can be written as product of Artin L-functions, and this suggests that Dedekind's conjecture is indeed true.\\
\newline  
There is a elliptic analogue of Artin's conjecture. Let $E$ be an elliptic curve over a number field $K$. Hasse and Weil defined an $L$-function, $L_K(s)$ attached to $E$, and conjectured this $L$-functions has analytic continuation to whole $\mathbb{C}$. Birch and Swinnerton-Dyer conjectured that this $L$-function extends to an entire function, satisfying a suitable functional equation, and at $s=1$ has a zero of order equal to the rank of the group $E(K)$ of K-rational points of $E$. As a elliptic analogue of Dedekind's conjecture one may ask, if $\frac{L_K(s)}{L_F(s)}$ is entire for a finite extension of number fields $K/F$. This we can expect to be true, because $\mbox{Rank}(E(K))\geq \mbox{Rank}(E(F))$. In the vein of methods used by Van der Waall and Brauer-Aramata, M. Murty and K. Murty showed the result holds when $K/F$ is a Galois extension or contained in a solvable Galois extension.  \\
\newline 
This note is organized as follows. In section 2, we introduce the preliminaries of representation theory if finite groups and Artin L-functions. In section 4 we present several variations on Dedekind's conjecture. We start with explaining Van der Waall's proof using Lemma 3.1.1. Based on this lemma, we deal with quotient of Artin L-functions, and prove Theorem 3.2.1, 3.2.2 and 3.2.4. For example, Theorem 3.2.2 asserts the quotient,
\[\frac{L(s, \mbox{Ind}_{H}^{G} (\psi), K/F)}{\prod_{\chi \in S_{\psi}} L(s,\chi, K/F)}\]
is entire, where $\psi$ is a one dimensional character on $H$, and $S_{\psi}$ is set of all one dimensional characters of $G$, whose restriction onto $H$ is $\psi$. This work was done by Murty and Raghuram \cite{MR} in 2000. They even predicted that, it is possible that an analogue result for any character of higher level holds, and proved a weaker result in Theorem 3.2.4 based on Corollary 3.1.3. We further explain how Lansky and Wilson generalized Theorem 3.2.4 for characters of higher levels, in 2005. As a corollary they showed the product  
\[\prod_{l(\chi)=i} L(s,\chi, K/F)^{\big(\chi, \mbox{Ind}_{H}^G(\psi_0)\big)}\] 
is entire, in Corollary 3.3.2. From which it is reasonable to expect product of all $L(s,\chi,K/F)$, with $l(\chi)=i$ is entire as well. $l(\chi)$ is defined to be minimum of all $i$ such that $\chi|_{G^i}=1_{G^i}$, where $G^i$ denotes $i^{th}$ derived subgroup of $G$, and $\psi_0$ be a one-dimensional character of a subgroup $H$ of $G$. Unfortunately we do not whether the result is true or not.\\ 
\newline 
In section 5 we study Artin L-functions more closely, using Heilbronn characters. In 1972 Heilbronn introduced a tracking device to relate the zeroes of an Dedekind zeta function over $K$ at a complex point $s_0$ to zeros of zeta functions attached to the subfields of $K$ at the same point $s_0$. As an evidence, we start with proving Aramata-Brauer's theorem in Theorem 4.1.1. We further see how does this simple idea has blossomed the area of research on Artin's holomorphy conjecture of L-functions. We do this by proving Theorem 4.1.2, 4.1.3 and 4.2.1. For example, while proving Theorem 4.2.1, we see the product 
\[\prod_{l(\chi)>i} L(s,\chi,K/F)^{\chi(1)}\]
is entire. The fact is, the above product can be identified with  $\frac{\zeta_K(s)}{\zeta_{K^i}(s)}$, where $K^i$ is the subfield of $K$ fixed by the $i^{th}$ derived subgroup of $\mbox{Gal}(K/F)$, shows that it is entire by Aramata Brauer's theorem in Theorem \ref{AB1}. In the same spirit of Heilbronn, Stark proved Artin's conjecture is locally true, in the sense, if $\mbox{ord}_{s=s_0}(\zeta_K(s)) \leq 1$, then all Artin L-functions are holomorphic at $s=s_0$. Which we record in Theorem 4.3.1. Inspired by Stark's work, Foote and Murty obtained a partial generalization of Stark's result for solvable Galois extensions. They proved, if $|G|=\prod_{i=1}^{k} p_i^{n_i}$ such that $p_1 <p_2<\cdots <p_k$ be distinct primes such that $\mbox{ord}_{s=s_0} (\zeta_K(s)) \leq p_2-2$, then all Artin L-functions $L(s, \chi, K/F)$'s are holomorphic at $s = s_0$. In the same spirit, Wong proved in 2017 that if $\mbox{ord}_{s=s_0} \Big(\zeta_K(s)/\zeta_F(s)\Big) \leq p_2-2$, then all Artin L-functions $L(s, \chi, K/F)$'s are holomorphic at $s = s_0$. We record Wong's proof in Theorem 4.3.4.\\
\newline 
In section 6, we start with setting a formalism for Heilbronn characters and apply this machinery to study L-functions. Firstly, we introduce the notion of weak arithmetic Heilbronn characters that satisfy properties analogous to some properties of the classical Heilbronn characters known by the work of Heilbronn-Stark (Theorem 5.1.2), Aramata-Brauer (Corollary 5.1.1), Foote-Murty (Theorem 5.4), and Murty-Murty (Theorem 5.1.3). In the section, we further put more conditions on weak arithmetic Heilbronn characters, which take them more closer to Heilbronn characters. This new formalism is meant to be used for several arithmetic usage, so we refer this as arithmetic Heilbronn characters. We start with proving Uchida-van der Waalls' theorem in this new arithmetic settings (Corollary 5.2.1), and as a sequel go on to derive several extensions of results by Murty-Raghuram, Lansky-Wilson for arithmetic Heilbronn characters. These are recorded from Theorem 5.4.2 to 5.4.3. Further we apply these results to study Artin-Hecke L-functions and L-functions of CM-elliptic curves. More precisely, we first discuss how R. Murty and K. Murty proved elliptic analogue of Brauer-Aramata theorem and Uchida Van-der Waall's theorem for CM elliptic curves in Theorem 5.5.2. Further we use results of section 5.1, 5.2, 5.3 and 5.4 to establish Theorem 5.5.6, 5.5.7 and 5.5.8.\\ 
\newline 
In section 6, we give a brief review of the supercharacter theory and study Artin L-functions via the theory of supercharacters. Here also we introduce to supercharacter-theoretic analogue of Heilbronn characters and prove Theorem 6.0.4 to show, Artin's conjecture for L-functions attached to supercharacters could be true locally. Furthermore, we introduce to product of supercharactertheories to see if Artin conjecture is true for Artin L-functions attached to suitable supercharacters.\\ 
\section{Preliminaries}

In this section, we will recall some notations and results from the theory of finite groups. We always denote $G$ by a group (unless otherwise specified) and $\{e_G\}$ or $e$ to be its identity element. Whenever I am writing $\{e\}$ while working with group $G$, it is meant $e=e_G$, the identity element of $G$. In this note $Z(G)$ will always stand for the centre of $G$. 
\begin{definition}[Solvable and Super Solvable Group] A finite group $G$ is said to be solvable if there are subgroups $\{1\} = G_0 \subset G_1 \subset \cdots \subset G_k = G$ such that $G_{j-1}$ is normal in $G_j$, and $G_j/G_{j-1}$ is an abelian group, for $j=1, 2,\cdots, k.$ And $G$ is said to be supersolvable if the quotients $G_j/G_{j-1}$s are cyclic. 
\end{definition}
\begin{rem} Supersolvability is stronger than the notion of solvability.
\end{rem}
\begin{definition}[Commutator and Derived Series] \label{Der} Let $[G,G]$ to be the subgroup of $G$ generated by all $xyx^{-1}y^{-1} \mid x,y \in G$. This is called the commutator subgroup of $G$. We denote a series of subgroups $\{G^i\}_{i \geq 1}$ by,
\[G^1=[G,G],G^2=[G^1,G^1], \cdots, G^i=[G^{i-1},G^{i-1}].\]
\end{definition}
\begin{rem} $G$ is solvable if and only if there exist a non negative integer $k$ such that $G^k=\{e_G\}.$
\end{rem}
\begin{definition}
Let $G$ be a finite group. A representation $\rho$ of $G$ on a finite dimensional vector space $V$ over $\mathbb{C}$ is a group homomorphism from $G$ to $GL(V)$, the group of all invertible linear maps on $V$. 
\end{definition}
In this note, we will restrict ourselves to complex representations, i.e., we will always assume $V$ is a vector space over $\mathbb{C}$. Given a representation $\rho$ of $G$, the character of $\rho$ is a function on $G$ defined by $\chi(g) = \mbox{tr} (\rho(g))$. A subspace $W$ of $V$ is called $G$-invariant if $\rho(g)w \in W$ for
all $g \in G$ and all $w \in W$. In this case, $\rho$ can be seen as a representation of $G$ on $W$, and we denote such a representation by $\rho|_{W}$. A representation $\rho$ is said to be irreducible if there is no proper and non-trivial $G$-invariant subspaces $W$ of $V$. In
addition, the character $\chi$ of a representation $\rho$ is called irreducible if $\rho$ is irreducible.
\begin{definition}[Inner Product] If $f_1,f_2 : G \to \mathbb{C}$ are two functions on $G$, one can define their inner product by
\[(f_1,f_2)=\frac{1}{|G|} \sum_{g \in G} f_1(g) \overline{f_2(g)}.\] 
\end{definition} 
If $f : G \to \mathbb{C}$ is constant on each conjugacy class in $G$, then $f$ is called a class function on $G$. We will let $C(G)$ denote the space of class functions of $G$. We denote $\mbox{Irr}(G)$ to be set of all irreducible characters of $G$. It can be showed that set of all irreducible characters forms an orthonormal basis for the inner product space of all class functions on $G$ with respect to the inner product defined above. Throughout the whole note, we will denote $1_G$ to be the trivial character by, $1_G(g)=1$ for all $g\in G$. We will also denote the regular character as
\[\mbox{Reg}_G=\sum_{\chi \in \mbox{Irr}(G)} \chi(1) \chi.\]
\begin{definition}[Length of Characters] \label{Len} Let $G$ be a solvable group, and $\chi$ be a character of $G$. We denote length $l(\chi)$ by, 
\[l(\chi)=\min\{i \mid \chi|_{G^i}=\chi(1)1_{G^i}\}.\]
\end{definition} 
\begin{rem}We assumed $G$ to be solvable, to ensure that such a $i$ exists at all.
\end{rem}
\begin{definition}[Induction and Inflation of Characters] \label{Inf} Let $H$ be a subgroup of $G$ and $f$ be a class function on $H$. Then the induction of $f$ from $H$ to $G$ is defined by
\[\emph{Ind}^G_H f(x) =\frac{1}{|H|} \sum_{g \in G}\widehat{f}(g^{-1}xg),\]
where $\widehat{f}$ is defined by, $\widehat{f}(g)=f(g)$ for all $g\in H$, and zero otherwise. Now suppose $H$ is normal in $G$ and $\chi$ be a character of $G/H$. We denote inflation of $\chi$ from $G/H$ to $G$ be,
\[\emph{Inf}_{G/H}^{G}(\chi)=\widetilde{\chi}=\chi \circ \pi,\]
where $\pi:G \rightarrow G/H$ is the canonical projection map. \end{definition}
\begin{rem} Throughout the note, we will denote inflation either by  $\mbox{Inf}_{G/H}^{G}(\chi)$ or $\widetilde{\chi}$.
\end{rem}
The theorems below are very standard results in theory of representation of finite groups.
\begin{theorem} [Frobenius Reciprocity] \label{FR}For all class functions $\phi$ on $H$, a subgroup of $G$, and all class functions $\chi$ of G,
\[(\emph{Ind}_{H}^{G} \phi,\chi)=(\phi,\chi|_{H}).\]
\end{theorem}
\begin{theorem}[Mackey, \cite{Serre}] \label{Mac} Let $\psi$ be a character of $H$ and let $K$ be a subgroup of $G$ then,
\[\emph{Ind}^G_H(\psi)|_K =\sum_{x \in K \backslash G/H} \emph{Ind}_{K \cap xHx^{-1}}^{G}(\psi^x),\]
where $\psi^x(g)=\psi(x^{-1}gx)$ for all $g\in xHx^{-1}\cap K$. Moreover if $HK=G$, we have,
\[\emph{Ind}^G_H(\psi)|_K=\emph{Ind}_{H \cap K}^{H}(\psi|_{H\cap K}).\]
\end{theorem}
\begin{definition}[Monomial character and M-group] We define a character of $G$ to be monomial if it is induced from a linear character (i.e., a character of degree $1$) and call $G$ be a $M$ group if all irreducible characters of $G$ are monomial.
\end{definition}
\begin{theorem}[Brauer Induction Theorem] Let $G$ be a finite group and $\chi$ be a character of $G$. Then there exist integers $n_i$'s such that
\[\chi=\sum_{i} n_i \emph{Ind}_{H_i}^{G} \psi_i,\]
where $\psi_i$'s are one dimensional characters of $H_i$'s
\end{theorem}
\begin{proof} See \cite{Serre}, pp. 78. 
\end{proof}
Now let us state Clifford's theorem.
\begin{theorem}[Clifford] Let $N$ be a normal subgroup of $G$. If $\chi$ is a complex irreducible character of $G$, and $\mu$ is an irreducible character of $N$ with $(\chi|_{N},\mu)\neq 0$, then
\[\chi|_{N}=e\Big(\sum_{i=1}^{t} \mu^{g_i}\Big),\]
where $e|[G:N]$ and $\mu^g$ is defined by, $\mu^{g}(n)=\mu(g^{-1}ng)$ for all $g\in G,n\in N$.
\end{theorem}
As a consequence we deduce the following result, which will be required for our purpose. 
\begin{theorem} \label{Cli}Let $N$ be a normal subgroup of $G$ with $[G : N] = p$, a prime. Then for any $\chi \in \emph{Irr}(G)$, either $\chi|_{N}$ is irreducible or $\chi|_{N}=\sum_{i=1}^{p} \phi_i$, where $\phi_i$'s are distinct, irreducible and $\emph{Ind}_{N}^{G} \phi_i=\chi$.
\end{theorem}
\begin{theorem}[Blichfeldt] \label{Bli} If $G$ admits a non-central abelian normal subgroup, then all faithful irreducible characters of $G$ are imprimitive, in other words induced from a proper subgroup. 
\end{theorem}
\begin{proof} See \cite{Bli}, Corollary 50.7
\end{proof}
\begin{theorem}[Ito] \label{Ito} Suppose that $G$ is solvable, and has a faithful irreducible character of degree less than $q-1$. Then $G$ admits an abelian normal Sylow $q-$subgroup
\end{theorem}
\begin{proof} See \cite{Feit}, pp. 128
\end{proof}
\begin{theorem}[Huppert] \label{Hup} Assume that $G$ is solvable, and that $G$ admits a normal subgroup $N$ such that $G/N$ is supersolvable. If all Sylow subgroups of $N$ are abelian, then $G$ is an $M$-group.
\end{theorem}
\begin{proof} See \cite{Hup}, Theorem 6.23
\end{proof}
Now we will introduce Artin L-functions and discuss some of their crucial properties. The Artin L-functions are the natural analogue of Dedekind zeta functions. One of the main purposes of Artin L-function is to understand more about Frobenius element. The Frobenius element is significant because it describes how prime ideals split when lifted above. Let $K/F$ be a finite Galois extension of algebraic number fields with Galois group $G$, and $V$ be a finite dimensional vector space over $\mathbb{C}$. Suppose we have a representation
$$\rho : G \to GL(V ),$$ 
where $GL(V )$ denotes the group of invertible linear transformations of $V$ into itself. Let $O_K$ and $O_F$ be the corresponding ring of integer of $K$ and $F$ respectively. Let $\mathfrak{p}$ be a prime ideal of $F$ and $\mathfrak{P}$ be a prime ideal of $K$ lying over $\mathfrak{p}$.  
\begin{definition} The decomposition group $D_{\mathfrak{P}}$, of $\mathfrak{P}$ is defined to be 
\[ D_{\mathfrak{P}}=\{\sigma \in G \mid \sigma(\mathfrak{P})=\mathfrak{P}\}.\] 
\end{definition} 
\noindent It can be showed that  $\mathfrak{f}_K=O_K/\mathfrak{P}$ is a cyclic Galois extension of $\mathfrak{f}_F=O_F/\mathfrak{p}$. Degree of this extension (say $f$) is called the residue degree of $\mathfrak{P}$ over $\mathfrak{p}$. One can show (see \cite{NK}) there exist a surjection from $D_{\mathfrak{P}}$ to $\mbox{Gal}(\mathfrak{f}_K/\mathfrak{f}_F)$ with kernel $I_{\mathfrak{P}}$, and call it inertia group of $\mathfrak{P}$. In other words,
\[I_{\mathfrak{P}}=\{\sigma \in G_{\mathfrak{P}} \mid \sigma(x)=x \mod \mathfrak{P}, \forall x \in O_K\}.\]
\noindent One can show  $\mathfrak{p}O_K=\Big(\prod_{i=1}^{r}\mathfrak{P_i}\Big)^{e}$, with $\mathfrak{P_1}=\mathfrak{P}$. Number $e$ is the cardinality of the inertia group $I_{\mathfrak{P}}$ at $\mathfrak{P}$. One can show all of $\mathfrak{P_i}$'s have same residue degree $f$ with $[K:F]=efr$. The number $e$ is called ramification index of $\mathfrak{p}$. Prime ideal $\mathfrak{p}$ is called unramified in $K$ if $e=1$, and otherwise it is called ramified. So if $\mathfrak{p}$ is unramified, 
\[D_{\mathfrak{P}} \simeq  \mbox{Gal}\Big(\mathfrak{f}_K/\mathfrak{f}_F\Big).\]
\noindent Let $\tau_{\mathfrak{P}}$ be a generator of the cyclic group $\mbox{Gal}\Big(\mathfrak{f}_K/\mathfrak{f}_F\Big)$. Then the Frobenius element $\sigma_{\mathfrak{P}}$ is defined to be the element of $D_{\mathfrak{P}}$ that gets mapped to $\tau_{\mathfrak{P}}$ under this isomorphism. Define the Frobenius substitution
attached to $\mathfrak{p}$ to be the conjugacy class $\sigma_{\mathfrak{p}} = \{\sigma_{\mathfrak{P}} \mid \mathfrak{P} \hspace{0.1cm} \text{lying over}  \hspace{0.1cm} \mathfrak{p}\}$. We define the unramified part to be 
\[L_{unr}(s, \rho, K/F)=\prod_{\mathfrak{p}}\det\Big(1-\rho(\sigma_{\mathfrak{p}})N(\mathfrak{p})^{-s}\Big)^{-1},\]
\noindent where the product runs over all unramified primes $\mathfrak{p}$ of $O_F$. Observe that this definition is well defined. Because  we replace the term $\sigma_{\mathfrak{p}}$ by any $\sigma_{\mathfrak{q}}$ for any $\mathfrak{q}$ dividing $\mathfrak{p}$, then, $\mathfrak{q}$ would produce a conjugate element $\sigma_{\mathfrak{q}}$ to $\sigma_{\mathfrak{p}}$. Fortunately elementary determinant property says, $\det(X) = \det(YXY^{-1})$. This was Artin's original definition avoiding the local factors at ramified primes. Later he defined local factors at ramified factors as follows,\\
\newline 
\noindent Let $\mathfrak{p}$ be a ramified prime and $\mathfrak{P}$ be a prime lying over it. Define $\sigma_{\mathfrak{P}}$ to be the element of the quotient $D_{\mathfrak{P}}/I_\mathfrak{P}$ that maps to the
generator $\tau_{\mathfrak{P}}$ of $\mbox{Gal}\Big(\mathfrak{f}_K/\mathfrak{f}_F\Big)$. Define
a new vector space
\[V^{I_{\mathfrak{P}}} = \{x \in V 
\mid \rho(\sigma)x=x, \forall \sigma \in I_{\mathfrak{P}}\}.\]
\noindent This is the subspace of $I_{\mathfrak{P}}$ invariant. We define the ramified part as
\[L_{ram}(s, \rho, K/F)=\prod_{\mathfrak{p}}\det\Big(1-\rho(\sigma_{\mathfrak{P}})N(\mathfrak{p})^{-s}|_{V^{I_{\mathfrak{P}}}}\Big)^{-1},\]
\noindent where the product runs over all ramified primes $\mathfrak{p}$ of $O_F$. In the above equation, $\rho(\sigma_{\mathfrak{P}})N(\mathfrak{p})^{-s}|_{V^{I_{\mathfrak{P}}}}$ does not depend
on the element in the coset of $\sigma_{\mathfrak{P}}$. This is because we are now working in the vector space $V^{I_{\mathfrak{P}}}$ which consists of elements fixed by $I_{\mathfrak{P}}$. Consequently, nothing is  changed if we shift by an element of $I_{\mathfrak{P}}$.
\begin{definition}[Artin L-function] Define Artin L-function $L(s, \rho, K/F)$ as
$$L(s, \rho, K/F)=L_{unr}(s, \rho, K/F)L_{ram}(s, \rho, K/F).$$ 
\end{definition} 
\noindent In this note we will work with $V=\mathbb{C}^n$. Note that there is a more explicit description of $L(s, \rho, K/F)$ using the logarithm. For an unramified prime $\mathfrak{p}$, let $\lambda_i(\mathfrak{p}),\cdots, \lambda_n(\mathfrak{p})$ be the eigenvalues of $\rho(\sigma_{\mathfrak{p}})$ and notice that
$$\det\big(I_n-\rho(\sigma_{\mathfrak{p}})N(\mathfrak{p})^{-s}\big)^{-1}=\prod_{i=1}^{n}\big(1-\lambda_i(\mathfrak{p})N(\mathfrak{p})^{-s}\big)^{-1},$$
from which convergence of the unramified factor in the the region $\mbox{Re}(s)>1$ is clear because, all of those  $\lambda_i(\mathfrak{p})$'s have modulus one, being root of unities (as $G=\mbox{Gal}(K/F)$ is finite). Moreover, the ramified part $L_{ram}(s,\rho,K/F)$ is a finite product of functions converging in $\mbox{Re}(s)>1$. And so the Artin L-function $L(s,\rho,K/F)$ converges in $\mbox{Re}(s)>1$. This observation led Artin to conjecture the following.\\
\begin{conjecture}[Artin, \cite{BAM}] \label{ACc} Every Artin L-function associated to an irreducible non trivial representation $\rho$ of $\emph{Gal}(K/F)$ admits an analytic continuation to the whole of the complex plane. 
\end{conjecture} 
We will talk about this in the later section. But first we shall note down some key properties of Artin L-functions, which will be used throughout the whole note. From next time onward we will denote $L(s,\rho,K/F)$ as $L(s,\chi,K/F)$ where $\chi(g)=\mbox{Tr}(\rho(g))$.
\begin{lemma} \label{first} Let $H$ be a subgroup of $G$, $\chi_1, \chi_2$ be characters of $G$, $\phi$ be a characters of $H$. Let $M$ be a normal extension over $F$ contained in $K$, and $\chi$ be a character of $\mbox{Gal}(M/F)$. Then we have the following
\begin{itemize}
\item $L(s, \chi_1 + \chi_2,K/F) = L(s, \chi_1, K/F)L(s,\chi_2, K/F)$
\item $L(s, \emph{Ind}^G_H \phi, K/F) = L(s, \phi, K/K^H)$ 
\item $L(s,\widetilde{\chi},K/F) = L(s,\chi,M/F)$
\end{itemize}
\noindent where $K^H $is a subfield of $K$ fixed by $H$.
\end{lemma}
\begin{proof} By the Dirichlet series expansion of $\log L(s, \chi,K/F)$, it is clear that 
\[\log L(s, \chi_1 + \chi_2,K/F) = \log L(s, \chi_1,K/F) + \log L(s, \chi_2,K/F).\] 
and first part follows. For second part, See \cite{thedude} pp. 222. and for third part, see \cite{32} pp. 404. 
\end{proof} 
We will now describe one dimensional Artin L-functions. We need some help of Class field theory here. Abelian reciprocity law allows us to interpret one dimensional Artin L-functions as a Hecke L-functions attached to some ray class group. Let us briefly describe.

A modulus for $F$ is a function
$$m:\{\mbox{primes of F}\} \to \mathbb{Z}$$
such that,\\  
\newline 
$(a)$ $m(\mathfrak{p})\geq 0$ for all primes $\mathfrak{p}$, and $m(\mathfrak{p})=0$ for all but finitely many $\mathfrak{p}$.\\
$(b)$ If $\mathfrak{p}$ is real, then $m(\mathfrak{p})=0$ or $1$.\\
$(c)$ if $\mathfrak{p}$ is complex, then $m(\mathfrak{p})=0$.\\
For definition of real and complex primes, see \cite{NK} pp. 183. Traditionally one writes a modulus as
\[\mathfrak{m}=\prod_{\mathfrak{p} \subseteq O_F} p^{m(\mathfrak{p})}.\]
A modulus $\mathfrak{m}$ can also be written as
$$\mathfrak{m}=\mathfrak{m}_{\infty}\mathfrak{m}_0,$$ 
where $\mathfrak{m}_{\infty}$ is a product of real primes and $\mathfrak{m}_0$ is product of positive powers of prime ideals, and hence can be identified with an ideal of $F$. For a modulus $\mathfrak{m}$, define $I^{\mathfrak{m}}_F$ to be set of fractional ideals of $K$ having no prime factor from $\mathfrak{m}$. A ray class group modulo $\mathfrak{m}$ is a quotient of $I^{\mathfrak{m}}_{F}$, for a more detailed definition, see \cite{NK} pp. 363. We denote this as $C^{\mathfrak{m}}_F$.
\begin{theorem}[Abelian Reciprocity Law] Let $K/F$ be a finite abelian extension of number fields and let $m$ be a modulus for $F$ divisible by all ramified primes. Then the Artin map (given by Artin symbol) $\phi^{\mathfrak{m}}_{K/F}:I^{\mathfrak{m}}_F \to \emph{Gal}(K/F)$ is surjective. Moreover, $\mbox{Gal}(K/F)$ is isomorphic to a quotient of $C^{\mathfrak{m}}_F$. 
\end{theorem} 
\begin{proof} First let us briefly describe what is Artin map. Clearly $I^{\mathfrak{m}}_F$ is generated by all but finitely many unramified prime ideals. For each such $\mathfrak{p}$, send it to $\sigma_{\mathfrak{p}}$ and generate to whole $I^{\mathfrak{m}}_{F}$ multiplicatively. Let $H \subset \mbox{Gal}(K/F)$ be the image of $\phi^{\mathfrak{m}}_{K/F}$ and let $L=K^H$ be its fixed field. For each prime $\mathfrak{p}\in I^{\mathfrak{m}}_F$ the automorphism $\phi^{\mathfrak{m}}_{K/F}(\mathfrak{p})$ acts trivially on $L$, which implies that $\sigma_{\mathfrak{p}}$ acts trivially on $L$. So from definition it follows that $\mathfrak{p}$ splits completely in $L$ (because, residue degree would be one). The group $I^{\mathfrak{m}}_F$ contains all but finitely many primes $\mathfrak{p}$ of $F$, so the Dirichlet density (will be defined in next section) of the set of primes of $F$ that splits completely in $L$ is $1$. It is a consequence of class field theory that, this density is in general $\frac{1}{[L:F]}$. And so we have $[L:F]=1$ in this case. Which implies $H=\mbox{Gal}(K/F)$, as desired. Under this surjective map given by Frobenius elements, one can get the desired isomorphism as well. \end{proof}
The result above is a consequence of Class field theory, due to Artin. But this theorem allows us to define a character $\psi$ of the class group by $\psi(\mathfrak{p}) =\chi(\sigma_{\mathfrak{p}})$ for $\mathfrak{p}$ relatively prime to $\mathfrak{m}$. We set, 
\[L(s, \chi,K/F) = L(s, \psi),\] 
where the latter function is the Hecke L-function \label{Hecke} attached to the ray class group modulo $\mathfrak{m}$. When $\chi$ is non trivial, it is known that the the Hecke L-function associated to $\chi$, have analytic continuation to whole plane. See \cite{lonewolf} pp. 16 for a discussion on this. By Abelian reciprocity law one can see, $\chi$ is trivial if and only if the corresponding Hecke character is trivial. So we have the following result,
\begin{theorem} \label{ODT}  For a one dimensional character $\chi,L(s,\chi,K/F)$ is entire for $\chi \neq 1_G$. And when $\chi=1_G, L(s,\chi,K/F)$ is holomorphic except at $s = 1$, where it has a pole. Moreover, $L(s,\chi,K/F)$ has no zero at $s=1$.
\end{theorem} 
Second part of theorem follows from \cite{lonewolf} pp. 27. We will use this result extensively throughout the whole note. 
\begin{rem} In general a Hecke-character, $\psi$ (associated to a modulus $\mathfrak{m}$) is defined to be a group homomorphism from the ray modulo $\mathfrak{m}$ to $\mathbb{C}$.\end{rem}
In analytic number theory, it is convenient to complete the L-function under consideration. This means we include gamma factors and a conductor term of the form $Q^s$ for some positive real number $Q$. In our case, the completed Artin L-function satisfy a nice functional equation relating the Artin L-function's value at the points s with the contragradient Artin L-function at the point $1-s$. We define the completed function $\Lambda(s,\chi,K/F)$ associated to the Artin L-function $L(s,\chi, K/F)$ as
$$\Lambda(s,\chi,K/F)=A(\chi)^{\frac{s}{2}}\gamma_{\chi}(s)L(s, \chi, K/F),$$ 
\noindent where $A(\chi)$ is some real number and $\gamma_{\chi}(s)$ is the product of certain gamma factors. (See \cite{NvL} pp. 28.)
\begin{theorem}[Functional Equation] $\Lambda(s, \chi,K/F)$ satisfies the functional equation 
\[\Lambda(1-s, \chi,K/F) = W(\chi)\Lambda(s, \overline{\chi},K/F),\]
for all $s \in \mathbb{C}$, where $W(\chi)$ is a number of absolute value one, and $\overline{\chi}$ is complex conjugate character of $\chi$. 
\end{theorem}
\begin{proof} See \cite{NvL} pp. 28 for more details.
\end{proof} 
Of course we do not know whether Artin's conjecture is true, but we can indeed get their meromorphic continuation, from (weak) Brauer's induction theorem. Which we are going to prove now.  
\begin{theorem}[Weak Brauer's induction theorem]For any irreducible $\chi$ of $G$, there are subgroups $\{H_i\}$, one-dimensional characters $\psi_i$ of $H_i$ and rational numbers $m_i$ with
\[\chi=\sum_{i} m_i \emph{Ind}_{H_i}^{G} \psi_i.\]
\end{theorem}
\begin{proof} For each subgroup $H$ of $G$ and each
irreducible character $\psi$ of $H$, define $a_H(\psi, \chi)$ by 
\[\mbox{Ind}^G_H \psi =\sum_{\chi \in \mbox{Irr}(G)}a_H(\psi, \chi)\chi.\]
For each $\chi$, let $A_{\chi}$ be the vector $\Big(a_H(\psi, \chi)\Big)$ as $H$ varies over all cyclic subgroups of $G$ and $\psi$ varies over all irreducible characters of $H$. Our claim is that $A_{\chi}$'s are linearly independent over $\mathbb{Q}$. Suppose not, so there exist $c_{\chi}$'s such that $\sum_{\chi \in \mbox{Irr}(G)} c_{\chi}A_{\chi}=0$. Consider $\phi=\sum c_{\chi}\chi$. Let $\psi$ be any irreducible character of $H$, then from Frobenius reciprocity it follows that $(\phi|_{H},\psi)=0$. And hence $\phi|_{H}=0$, for any cyclic subgroup $H$ of $G$. For any $g\in G$, taking $H=\langle g \rangle$ we see $\phi \equiv 0$. Since $A_{\chi}$'s are linearly independent over $\mathbb{Q}$, we can choose a set of cyclic subgroups $H_i$ and characters $\psi_i$ of $H_i$ the matrix 
\[\Big(a_{H_i} (\psi_i, \chi)\Big)\]  
is non-singular, if not then we would get a equation of form, $\sum_{\chi \in \mbox{Irr}(G)} c_{\chi}A_{\chi}=0$. Inverting the matrix above we get the desired result. 
\end{proof} 
\begin{theorem} \label{mero} Artin L-function $L(s, \chi,K/F)$ attached to an irreducible character $\chi$ admits a meromorphic continuation to whole $\mathbb{C}$.
\end{theorem}
\begin{proof} Using Lemma \ref{first} and (Weak) Brauer's induction theorem we can write \[L(s,\chi,K/F)=\prod_{i}L\big(s,\psi_i,K/K^{H^i}\big)^{m_i},\]
\noindent where $\psi_i$'s are one dimensional characters of $H_i$'s and $m_i$'s are rational numbers. We know from Theorem \ref{ODT} that all of $L\big(s,\psi_i,K/K^{H^i}\big)$'s are meromorphic. And so some power of $L(s,\chi,K/F)$ is meromorphic (in fact, Foote and K.Murty showed in \cite{FM} that the power is exactly $\chi(1)$). But then of course, $L(s,\chi,K/F)$ is meromorphic.
\end{proof}
From next section onward, we will discuss impact of character theory on quotient of L-functions.
\section{Holomorphy of quotients of certain Artin L-functions}
First let us define Dedekind zeta function associated to a number field $K$. It is defined to be, 
\[\zeta_{K}(s)=\sum\frac{1}{N(\mathfrak{a})^s}\]
in $\mbox{Re}(s)>1$, where $\mathfrak{a}$ runs over all non zero ideals in $O_K$. And $N(\mathfrak{a})$ is index of ideal $\mathfrak{a}$ in $O_K$. We consider the following conjecture made by Dedekind.
\begin{conjecture}[Dedekind] Let $K/F$ be a finite extension of number fields. Then $\zeta_K(s)/\zeta_F(s)$ is entire.
\end{conjecture}
\noindent Artin's conjecture has many consequences, one of them is, it shows Dedekind's conjecture to be true for the case of Galois extension $K/F$. 
\begin{theorem}[Artin's conjecture implies Dedekind's conjecture for Galois case] \label{AC to DC}
If $K/F$ is a finite Galois extension of number fields, then $\frac{\zeta_K(s)}{\zeta_F(s)}$ is entire if Artin's conjecture is true for any non trivial irreducible representation of $\rho$ of $\emph{Gal}(K/F)$. 
\end{theorem} 
\begin{proof} Note that 
\begin{align}
\prod_{\chi \hspace{0.2cm}\mbox{irreducible}} L(s,\chi,K/F)^{\chi(1)} &= L(s,\sum_{\chi \in \mbox{Irr}(G)} \chi(1) \chi, K/F)\\
&=L(s,\mbox{Reg}_G,K/F)\\
& = L(s,\mbox{Ind}_{\langle e_G \rangle}^{G}1_{G},K/F),
\end{align} 
\noindent and from Lemma \ref{first} it follows that, later one is same as $L(s,1_G,K)=\zeta_K(s)$. Above decomposition is known as Artin-Takagi decomposition. When $\chi$ is trivial character, $L(s,\chi,K/F)=\zeta_F(s)$. So we can write 
\[\zeta_F(s) \prod_{\chi \neq 1} L(s,\chi,K/F)^{\chi(1)}=\zeta_K(s),\]
and so clearly in case of Galois extension, Artin's conjecture implies $\zeta_{K}(s)/\zeta_F(s)$ is entire, as $\chi(1)>0$ for any $\chi \in \mbox{Irr}(G)$.
\end{proof}
\noindent In the next section, we will establish the result above unconditionally, i.e. without assuming Artin's conjecture. But utilizing the proof of Theorem \ref{AC to DC}, we can prove $\frac{\zeta_K(s)}{\zeta_F(s)}$ to be entire for any finite extension $K/F$ of number fields.
\begin{theorem}[Artin's conjecture implies Dedekind's conjecture] If $K/F$ is a finite extension of number fields, then $\frac{\zeta_K(s)}{\zeta_F(s)}$ is entire if Artin's conjecture is true for any non trivial irreducible representation of $\rho$ of $\emph{Gal}(K/F)$. 
\end{theorem} 
\begin{proof} 
Let $\overline{K}$ be normal closure of $K$ over $F$. Say $G=\mbox{Gal}(\overline{K} /F)$ and $H=\mbox{Ga}l(\overline{K} /K)$. Clearly both of $G$ and $H$ are finite Galois extensions. Yes, we will take advantage of this. First of all note that,
$$L(s, \mbox{Ind}_H^G(1_H),\overline{K}/K)=L(s, 1_{H}, \overline{K}/K)=\zeta_K(s).$$
\noindent We can write 
\[\mbox{Ind}_H^G(1_H)=1_G+\sum_{\chi \in \mbox{Irr}(G),\chi \neq 1} a_{\chi}\chi.\]
\noindent From Frobenius reciprocity (Theorem \ref{FR}) it follows, $1_G$ is coming with multiplicity one. And also $a_{\chi}$'s are non-negative as they are multiplicity of $\chi$ in $\mbox{Ind}_H^G(1_H)$. So we get
\[L(s, \mbox{Ind}_H^G(1_H),\overline{K}/K)=\zeta_F(s)\prod_{\chi \neq 1}L(s,\chi,K)^{a_{\chi}} =\zeta_K(s),\]
where $a_{\chi}$'s are all non negative. Artin's conjecture does rest of the job.
\end{proof} 
\subsection{First unconditional approach}
\noindent Artin's conjecture is a very strong statement and it is still not known to be true for most of the cases when it is attached to more then 2-dimensional representation. So we would like to investigate quotient of zeta functions in a different way. The result below was proved independently by Uchida \cite{Uchida} and Van der Waall \cite{Waalll} in 1975. We explain the later approach below. 
\begin{theorem}[Uchida, Van der Waall] \label{UVdW} Let $K/F$ be an extension of number fields and $\overline{K}$ be normal closure of $K$ over $F$. Suppose that $\emph{Gal}(\overline{K}/F)$ is solvable. Then $\zeta_K(s)/\zeta_F(s)$ is entire.
\end{theorem}
\noindent To prove the theorem we need the following lemma.
\begin{lemma} \label{imp} Let $G$ be a finite group and $H$ be a subgroup of $G$ such that, there exist an abelian normal subgroup $A$ with $H \cap A=\{e_G\}$ and $G=HA$. Let $\chi$ be an irreducible character of $H$. We can write
\[\emph{Ind}_{H}^{G}\chi =\sum \emph{Ind}_{H_i}^G \chi_i,\] 
where $\chi_i$'s are one dimensional characters of subgroups $H_i$ of $G$. 
\end{lemma} 
\begin{proof} Write 
\[\mbox{Ind}^G_H(1_H) = \sum m_{\chi}\chi,\]
where $0 \leq m_{\chi} \in \mathbb{Z}, m_1 = 1$ and $\chi$ ranges over all irreducible characters of $G$. Note that $m_1=1$ comes from Frobenius reciprocity, i.e coefficient of trivial character $1_G$ in $\mbox{Ind}_{H}^{G}(1_H)$ is one. From Mackey's theorem (Theorem \ref{Mac}) we have,
\[\mbox{Ind}^G_H(1_H)|_{A} = \mbox{Ind}^A_{H\cap A}(1_H|_{H \cap A}) = \mbox{Ind}^A_{\{1\}} 1 = \mbox{Reg}_A,\]
\noindent So,
\[\mbox{Ind}^G_H(1_H)|_A = \sum_{\epsilon \in \mbox{Irr}(A)} \epsilon=\sum m_{\chi}\chi|_A,\]
Therefore $m_{\chi} = 0$ or $1$ and $(\chi|_A,\epsilon) = 0$ or $1$, for any $\epsilon\in \mbox{Irr}(A)$ and $\chi \in \mbox{Irr}(G)$. Now take an $\epsilon \in \mbox{Irr}(A)$ such that there exist a $\chi \in \mbox{Irr}(G)$ with $m_{\chi} \neq 0$ and $(\chi|_{A},\epsilon)=1$. We define $T_{\epsilon}$ to be the inertia group of $\epsilon$ as 
\[T_{\epsilon}=\{\sigma \in G \mid \epsilon(\sigma a \sigma^{-1})=\epsilon(a) \mid \forall a \in A\}.\]
Clearly $A$ is contained in $T_{\epsilon}$ and taking $H_{\epsilon}=H \cap T_{\epsilon}$ we see $T_{\epsilon}=H_{\epsilon}A$. Fortunately as $H_\epsilon \cap A=\{e\}$, (because $H \cap A=\{e\}$) we can extend $\epsilon$ to a character $\hat{\epsilon}$ on $T_{\epsilon}$ by, $\hat{\epsilon}(xa)=\hat{\epsilon}(x)$ for all $x \in H_{\epsilon},a\in A$. 
We can write,  
\begin{equation}
    \mbox{Ind}_A^{T_{\epsilon}} \epsilon = \sum_{\psi \in \mbox{Irr}(T_{\epsilon})} m_{\psi}\psi. 
\end{equation}
Thus,
\[[T_{\epsilon}:A]\epsilon=(\mbox{Ind}_{A}^{T_{\epsilon}}\epsilon)|_{A}=\sum_{\chi \in \mbox{Irr}(T_{\epsilon})} m_{\psi}\psi|_{A}.\]
\noindent Therefore we can say for every $\psi \in \mbox{Irr}(T_{\epsilon}), \psi|_{A}$ is a multiple of $\epsilon$ and $m_{\psi}=(\psi|_{A},\epsilon)$. 
\begin{equation} 
\mbox{So,} \hspace{0.2cm} \psi|_{A}=m_{\psi} \epsilon,\hspace{0.1cm} \mbox{and hence} \hspace{0.2 cm} [T_{\epsilon}:A]=\sum m^2_{\psi}.
    \end{equation} 
It is clear that \[[T_{\epsilon}:A]=\sum m^2_{\chi}=(\mbox{Ind}_A^{T_{\epsilon}}\epsilon,\mbox{Ind}_A^{T_{\epsilon}}\epsilon) \leq (\mbox{Ind}_A^{G}\epsilon,\mbox{Ind}_A^{G}\epsilon),\]
and we can write
\[(\mbox{Ind}_A^{G}\epsilon)|_{A}=\frac{1}{|A|} \sum_{x \in G} \epsilon^{x}|_{A}=[T_{\epsilon}:A]\sum_{g \in G/T_{\epsilon}} \epsilon^{g},\]
and so, 
\begin{equation}
(\mbox{Ind}_A^{G}\epsilon,\mbox{Ind}_A^{G}\epsilon)=(\epsilon,(\mbox{Ind}_A^{G}\epsilon)|_{A})=[T_{\epsilon}:A].
\end{equation}
From (3.4) we can write, 
\[\mbox{Ind}_{A}^{G} \epsilon=\sum_{\psi \in \mbox{Irr}(T_{\epsilon})} m_{\psi}\mbox{Ind}_{T_{\epsilon}}^{G} \psi,\]
but from the observation made at (3.6), it follows that $\{\mbox{Ind}_{T_{\epsilon}}^{G} \psi\}_{\psi \in \mbox{Irr}(T_{\epsilon})}$ forms a pairwise distinct family of irreducible representation of $G$. Recall, we picked $\chi \in \mbox{Irr}(G)$ such that $(\chi|_{A},\epsilon)=1$. And so,
\[1=(\chi, \mbox{Ind}_{A}^{G} \epsilon)=\sum m_{\psi} (\chi, \mbox{Ind}_{T_{\epsilon}}^{G} \psi),\]
so there exist a unique $\phi \in \mbox{Irr}(T_{\epsilon})$ such that $(\chi, \mbox{Ind}_{T_{\epsilon}}^{G} \phi)=1$ and $m_{\phi}=1$. But then $\chi=\mbox{Ind}_{T_{\epsilon}}^{G} \phi$, as both of them are irreducible. So now it enough to show that $\phi$ is 1-dimensional. From (3.5) we have $\phi(1)=m_{\phi}\epsilon(1)=m_{\phi}$, as $\epsilon(1)=1$.  
\end{proof}
\textit{Proof of Theorem \ref{UVdW}.} We set $G = \mbox{Gal}(\overline{K}/F)$ and $H = \mbox{Gal}(\overline{K}/K)$. We may assume that $H$ is a maximal subgroup of $G$, because if $J$ is a maximal subgroup of $G$ with $H \subset J \subset G$, and $M$ is the fixed field of $J$, then we can write 
\[\frac{\zeta_K(s)}{\zeta_F (s)}=\Big(\frac{\zeta_K(s)}{\zeta_M(s)}\Big)\Big(\frac{\zeta_M(s)}{\zeta_F(s)}\Big).\]
By induction we may assume first factor is entire. So at any point $s=s_0$ we have
\[\mbox{ord}_{s=s_0}\Big(\frac{\zeta_K(s)}{\zeta_F (s)}\Big) \geq \mbox{ord}_{s=s_0}\Big(\frac{\zeta_M(s)}{\zeta_F(s)}\Big),\] and hence it is enough to show the second factor $\Big(\frac{\zeta_M(s)}{\zeta_F(s)}\Big)$ to be entire. Thus we may assume $H$ is maximal. Suppose $H$ contains a non trivial proper normal subgroup of $G$. But that would correspond to an extension $L$ of $K$ contained in $\overline{K}$ such that $\overline{K}/L$ is a normal extension, and that means $L/F$ is a Galois extension, which means $L=\overline{K}$. Contradiction to the fact that $H$ is non trivial. We know $G$ is finite and solvable, so there exist a minimal normal subgroup $A$ contained in $H$. 
\begin{claim} \label{yea} $A$ is abelian.
\end{claim}
\textit{Proof of the claim:} We know that $G$ is solvable, and $A\unlhd G$ implies that $A$ is also solvable in $G$. Now, $[A,A]$ is a characteristic subgroup of $A$, and thus normal in $G$. By minimality assumption, $A=[A,A]$ or $\{1\}=[A,A]$. If the former is true then $A$ would not be solvable, and so the latter must happen, which says precisely that $A$ is abelian.\\ 
\newline
\noindent We already argued before that $A$ is not contained in $H$, so by maximality of $H$ we must have $HA=G$. Now $H \cap A$ is also a normal subgroup of $G$, but by minimality of $A$, we must have $H \cap A=\{e_G\}$. Proof now follows from lemma \ref{imp} 
\begin{corollary} \label{imp1} Let $G$ be a finite group and $H$ be a subgroup of $G$ such that, there exist an abelian normal subgroup $A$ with $G=HA$. Let $\chi$ be an irreducible character of $H$. Then we can write
\[\emph{Ind}_{H}^{G}\chi=\sum \emph{Ind}_{H_i}^G \chi_i,\]
where $\chi_i$'s are one dimensional characters of subgroups $H_i$ of $G$.
\end{corollary} 
\begin{proof} Proof is similar to the proof of Lemma \ref{imp}. Using Mackey's theorem we get
\[\mbox{Ind}^G_H(1_H)|_{A}=\mbox{Ind}_{H \cap A}^{A}(1_{H\cap A})=\sum_{\epsilon|_{H\cap A}=1_{H\cap A},\epsilon \in \mbox{Irr}(A)} \epsilon.\]
Therefore we are only concerned with characters of $A$, which are identity on $H\cap A$. This is why we are again fortunate enough to lift those on $T_{\epsilon}$. Arguing as in Lemma \ref{imp}, the proof follows.
\end{proof}
\begin{corollary} \label{imp3} Let $G$ be a finite solvable group and $H$ be a subgroup then,
\[\emph{Ind}_{H}^{G}(1_H)=1_G+\sum \emph{Ind}_{H_i}^{G} \chi_i,\]
where $\chi_i$'s are one dimensional characters of subgroups $H_i$'s of $G$.
\end{corollary}
\begin{proof} First suppose there exist a non trivial normal subgroup $A$ of $H$. Then by induction hypothesis we can write
\[\mbox{Ind}_{H/A}^{G/A}(1_{H/A})=1_{G/A}+\sum \mbox{Ind}_{H_i/A}^{G/A} \chi_i,\]
\noindent where $H_i$'s are some subgroups of $G$ containing $A$. Now proof is just matter of fact after inflating both sides to $G$ (recall definition of inflation from \ref{Inf}). To be precise, we just need to observe 
\[\widetilde{\mbox{Ind}_{H/A}^{G/A}(1_{H/A})}=\mbox{Ind}_{H}^{G} 1_H, \hspace{0.1cm} \widetilde{1_{G/A}}=1_G, \hspace{0.1cm} \widetilde{\mbox{Ind}_{H_i/A}^{G/A} \chi_i}=\mbox{Ind}_{H_i}^{G}\widetilde{\chi_i},\]
which one can check by hands. So me may assume $H$ does not contain any non trivial normal subgroup. Since $G$ is solvable, it must contain a abelian normal subgroup $A$. If $HA=G$ then we are done by lemma \ref{imp}. So we may assume $HA$ is a proper subgroup of $G$. Again by induction hypothesis we can write, 
\[\mbox{Ind}_{H}^{HA}(1_H)=1_{HA}+\sum  \mbox{Ind}_{H_i}^{HA} \chi_i,\]
and our proof is complete inducing both sides to $G$.
\end{proof}
\begin{corollary}\label{imp2} Let $G$ be a finite solvable group and $H$ be a subgroup. If we denote $\{G^i\}$ to be the derived series of $G$, then for any $i>0$ we have
\[\emph{Ind}_{H}^{G}(1_H)=\emph{Ind}_{HG^i}^{G}1_{HG^i}+\sum \emph{Ind}_{H_j}^{G} \chi_j,\]
where $\chi_j$'s are one dimensional characters of subgroups $H_j$'s of $G$.
\end{corollary} 
\begin{proof} Let $k$ be the smallest integer such that $G^{k+1}=\{e\}$. Such a $k$ exists because $G$ is solvable. We consider the subgroup $H_1=HG^k$ of $G$. We will use double induction on $|G|$ and $|G/H|$.
First of all we may assume $i \leq k$, otherwise we are done. And in this case $G^k \subseteq G^i$. Suppose $H_1=H$, so $G^k$ is subgroup of both $G$ and $H_1$. So by induction hypothesis we can write,
\[\mbox{Ind}_{H/G^k}^{G/G^k}1_{H/G^k}=1_{G/G^k}+\sum_{j>1} \mbox{Ind}_{H_j/G^k}^{G/G^k}(\psi_j),\]
Inflating both sides with respect to the canonical projection $G \to G/G^k$ we are done for this case. Now If $H_1=G$, then we are done by Corollary \ref{imp1}. \\
\newline
So we may assume $H \subset H_1 \subset G$. Then using Corollary \ref{imp3} we can write
\[\mbox{Ind}_{H}^{H_1}(1_H)=1_{H_1}+\psi,\]
where $\psi$ is a sum of monomial character of $H_1$. Inducing both sides to ${G}$ we get, 
\[\mbox{Ind}_{H}^{G}1_{H}=\mbox{Ind}_{H_1}^{G}(1_{H_1})+\psi,\]
where $\psi'$ is a sum of monomial characters of $G$. Note that $HG^i \subseteq H_1G^i$. But on the other hand $H_1G^i=HG^k \subseteq HG^i$. So $H_1G^i=HG^i$, and the proof is therefore complete using induction on $\mbox{Ind}_{H_1}^{G}(1_{H_1})$, as $H_1G^i=HG^i$ for any $i$. 
\end{proof} 
\noindent Interesting part is we can apply these results to study quotient of various Artin L-functions, not just of Dedekind-zeta functions. This we will discuss in next section.
\subsection{Quotient of various Artin L-functions}
Following theorem is a generalization of Uchida-Van der Waall's result (Theorem \ref{UVdW}). We shall follow the following notion:\\ 
\newline
\noindent Let $\psi$ be some character of a subgroups $H$ of $G$. We call $\psi \otimes \phi$ to be twist of $\psi$ by $\phi$, where $\phi$ is some character of $H$ and $(\psi \otimes \phi)(h) = \psi(h)\phi(h)$ for all $h \in H$. One can check from definition that tensoring commutes with induction. In other words, 
\[(\mbox{Ind}_{H}^{G} \phi) \otimes \chi =\mbox{Ind}_{H}^{G} (\phi \otimes \chi|_{H}).\]
\begin{theorem}[Murty-Raghuram, \cite{MR}] \label{RR1} Let $K/F$ is a solvable extension of number fields with Galois group $G$. Let $\chi$ be a one dimensional character of $G$. Then for any subgroup $H$ of $G$ and for any one-dimensional character $\psi$ of $H$, if $m(\chi,\psi)=(\chi, \emph{Ind}_{H}^{G} \psi)$, then the following quotient
\[\frac{L(s, \emph{Ind}_{H}^{G} (\psi), K/F)}{L(s,\chi, K/F)^{m(\chi,\psi)}}\]
is holomorphic except at $s= 1$.
\end{theorem}
\begin{proof} First of all if $m(\chi,\psi)=0$, the result follows immediately from Theorem \ref{ODT} and Lemma \ref{imp}. But if $m(\chi,\psi)\neq 0$, then $m(\chi,\psi)=1$ (and in fact, $\chi|_{H}=\psi$) as both of $\chi,\psi$ are one dimensional and,
\[m(\chi,\psi)=(\chi, \mbox{Ind}_{H}^{G} \psi)=(\chi|_{H},\psi).\] 
From Corollary \ref{imp3} we have, 
\[\mbox{Ind}_{H}^{G}(1_H)=1_G+\sum \mbox{Ind}_{H_i}^{G} \chi_i,\]
where $\chi_i$'s are one-dimensional characters of $H_i$. Twisting both sides with $\chi$ and using the fact that twisting commutes with induction, we observe
\[\mbox{Ind}_{H}^{G}(\psi)=\chi+\sum\mbox{Ind}_{H_i}^{G}(\chi|_{H_i}\otimes \chi_i).\]
Note that all of $\chi_i$'s are one dimensional. So, all of $\chi|_{H_i}\otimes \chi_i$'s are one dimensional as well. Now proof follows trivially applying $L(s,\cdot,K/F)$ to both sides and using Theorem \ref{ODT}.
\end{proof}  
\noindent The main obstacle in above proof was the case when $\chi|_{H}=\psi$. Based on this observation we have a general result involving all irreducible characters $\chi$ lying over $\psi$. More precisely we have the following theorem.
\begin{theorem}[Murty-Raghuram, \cite{MR}] \label{RR2} Let $G$ be the Galois group of solvable Galois extension $K/F$. Let $H$ be a subgroup of $G$ and $\psi$ be a one dimensional character of $H$. Let $S_{\psi}$ be set of all one dimensional characters of $G$ whose restriction to $H$ is $\psi$. Then,
\[\frac{L(s, \emph{Ind}_{H}^{G} (\psi), K/F)}{\prod_{\chi \in S_{\psi}} L(s,\chi, K/F)}\]
is entire.
\end{theorem} 
\begin{proof} Proof is a bit tricky one. First suppose $S_{\psi}$ is empty. So $\psi$ can not be $1_{H}$, because $1_{G}|_{H}=1_{H}$. Frobenius reciprocity shows that $\mbox{Ind}_{H}^{G}(\psi)$ can not contain $1_G$, and so the proof follows from Theorem \ref{ODT}.\\
\noindent If $S_{\psi}$ is not empty, take one element $\chi_0 \in S_{\psi}$. From Corollary \ref{imp2} we can write
\[\mbox{Ind}_{H}^{G}(1_H)=\mbox{Ind}_{HG^1}^{G}1_{HG^1}+\sum \mbox{Ind}_{H_j}^{G} \chi_j,\]
\noindent where $\chi_j$'s are one dimensional characters of $H_j$'s. Twisting both sides with $\chi_0$ (and using the fact that twisting commutes with induction) we get
\[\mbox{Ind}_{H}^{G}(\psi)=\mbox{Ind}_{HG^1}^{G}(\chi_{0}|_{HG^1})+\sum \mbox{Ind}_{H_j}^{G} (\chi_j\otimes \chi_0|_{H_j}).\]
Note that $\sum \mbox{Ind}_{H_j}^{G} (\chi_j\otimes \chi_0|_{H_j})$ does not contain $1_G$, because both of $\mbox{Ind}_{H}^{G}\psi$ and $\mbox{Ind}_{HG^1}^{G}(\chi_{0}|_{HG^1})$ contains $1_G$ with same coefficient (by Frobenius reciprocity). So from Theorem \ref{ODT}, we can say $L(s,\sum \mbox{Ind}_{H_j}^{G} (\chi_j\otimes \chi_0|_{H_j}),K/F)$ is entire. Therefore it is now enough to show
\[\mbox{Ind}_{HG^1}^{G}(\chi_{0}|_{HG^1}) =\sum_{\chi \in S_{\psi}} \chi.\]
By Frobenius reciprocity one can see each $\chi \in S_{\psi}$ comes with coefficient one in \[\mbox{Ind}_{HG^1}^{G}(\chi_{0}|_{HG^1}).\] Now dimension of $\mbox{Ind}_{HG^1}^{G}\chi_{0}|_{HG^1}$ is $[G/HG^1]$. So proof is complete if we can show $|S_{\psi}|=[G/HG^1]$. First let us count how many one dimensional characters are there for $H$ and $G$. 
\begin{claim} For any subgroup $H$, number of 1-dimensional characters of $H$ are $|H/(H\cap G^1)|$ 
\end{claim} 
\textit{Proof of the claim:} Proof is not difficult. Point is that, any one dimensional character $\phi$ of $H$ factors through $H/(H\cap G^1)$. And this gives a one-one correspondence between 1-dimensional characters of $H$ and and of $H/(H\cap G^1)$. But $H/(H\cap G^1)$ is abelian, and so it has $|H/(H\cap G^1)|$ many 1-dimensional characters.\\ 
\newline
\noindent Note, for any one dimensional character $\psi$ on $H, \chi \in S_{\psi}$ if and only if $\mbox{Ind}_{HG^1}^{G}(\chi_{0}|_{HG^1})$ contains $\chi$ with coefficient one. We get this using Frobenius reciprocity and the fact that dimension of $\chi_{0}|_{HG^1}$ is one. This observation leads us to conclude $[G/HG^1]$ is upper bound for all $|S_{\psi}|$'s. Because, $\mbox{Ind}_{HG^1}^{G}(\chi_{0}|_{HG^1})$ has dimension $[G/HG^1]$. On the other hand, it is clear that $\bigcup_{\psi}S_{\psi}$ is set of all 1-dimensional charters of $G$ ($\psi$ runs over all 1-dimensional characters of $H$). But then 
\[[G/G^1]=|\bigcup S_{\psi}| \leq [G/HG^1]\times[H/(H\cap G^1)]=[G/G^1],\] 
and so we must have $|S_{\psi}|=[G/HG^1]$, as desired.
\end{proof} 
\begin{rem} \label{rem 1}
\noindent From Theorem \ref{RR2} Murty and Raghuram conjectured that, it is possible to divide $L(s,\mbox{Ind}_{H}^G(\psi),K/F)$ by the product of L-functions attached to all irreducible characters of level at most $i$ and get entire function. Note that only a special case of Corollary \ref{imp2} was used in the proof of Theorem \ref{RR2}. So using full power of this result there should be a generalization for any $i$. They proved this result partially. Later in 2005, Lansky and Wilson proved their conjecture. 
\end{rem}
We will come back to this in Section 3.3. For now let us discuss what Murty and Raghuram proved in \cite{MR},
\begin{theorem}[Murty-Raghuram, \cite{MR}] \label{RR3} Let $G$ be the Galois group of solvable Galois extension $K/F$ and $i\geq 1$ be a integer. Let $H$ be a subgroup of $G$ and $\psi$ be a one dimensional character of $H$. Let $S_{\psi}$ be set of all one dimensional characters of $G$ whose restriction to $H$ is $\psi$. Let $\psi'$ be the character obtained from $\psi$ extending it to $HG^i$. (We can do this because, being one-dimensional, $\psi$ is trivial on $H \cap G^i$). If $S_{\psi}$ is non empty then, 
\[\frac{L(s, \emph{Ind}_{H}^{G} (\psi), K/F)}{\prod_{l(\chi) \leq i} L(s,\chi, K/F)^{(\chi,\emph{Ind}_{H}^{G}\psi)}}\]
is entire.
\end{theorem}
\begin{proof} From Corollary \ref{imp2} we can write
\[\mbox{Ind}_{H}^{G}(1_{H})=\mbox{Ind}_{HG^i}^{G}(1_{HG^i})+\sum \mbox{Ind}_{H_j}(\psi_j).\]
As per assumption $S_{\psi}$ is non empty. Take $\chi_0 \in S_{\psi}$ and twisting both sides of above equation with $\chi_0$ we get
\[\mbox{Ind}_{H}^{G}(\psi)=\mbox{Ind}_{HG^i}^{G}(\psi')+\chi,\]
where $\chi$ is a sum of monomial characters of $G$ not containing $1_G$, because both of $\mbox{Ind}_{H}^{G}(\psi)$ and $\mbox{Ind}_{HG^i}^{G}(\psi')$ contains $1_G$ with same coefficient. Proof is now complete using Theorem \ref{ODT} and Lemma \ref{imp for i}, which we will show in the next section.
\end{proof} 
\subsection{Further work of Lansky-Wilson on Murty-Raghuram's conjecture}
In this section we shall discuss the generalization of Theorem \ref{RR3}. Let us go back to the discussion at Remark \ref{rem 1} and recall the result predicted by Murty and Raghuram. They proved (Theorem \ref{RR3}) their conjecture when $\psi$ extends to a character of $G$. We deal with the general case in this section, following the work by Lansky and Wilson \cite{LW}. First let us prove an important result of this section.
\begin{lemma} \label{imp for i} Let $G$ be a solvable group and $\psi$ be a one-dimensional character of a subgroup $H$ of $G$ such that $\psi|_{H \cap G^i}$ is trivial. Let $\psi'$ be the unique extension of $\psi$ to a character of $HG^i$, that is trivial on $G^i$. Then for any irreducible character $\chi$ of G, 
\[
    (\chi,\emph{Ind}^G_{HG^i}(\psi')) =\left\{ \begin{array}{lr}
        (\chi,\emph{Ind}^G_{H}(\psi)) & \text{if } l(\chi) \leq i\\
         0 & \text{if } l(\chi)>i
        \end{array} \right.
  \]
where $l(\chi)$ is defined to be length of character $\chi$, as defined in Definition \ref{Len}. And consequently,
\[\emph{Ind}_{HG^i}^{G} \psi'=\sum_{l(\chi) \leq i} (\chi,\emph{Ind}_{H}^{G}(\psi))\chi.\]
\end{lemma}
\begin{proof} If there exist a character $\chi$ of level more than $i$ such that $(\mbox{Ind}_{HG^i}^{G}(\psi'),\chi) \neq 0$, then $\mbox{Ind}_{HG^i}^{G}(\psi')$ contains $\chi$. But if we know $\mbox{Ind}_{HG^i}^{G}(\psi')|_{G^i}$ is multiple of $1_{G^i}$, then $\chi|_{G^i}$ must also be a multiple of $1_{G^i}$, a contradiction as $l(\chi)>i$. To show $(\mbox{Ind}_{HG^i}^{G}(\psi'),\chi) = 0,$ it is enough to show $\mbox{Ind}_{HG^i}^{G}(\psi')$ is a multiple of trivial character $1_{G^i}$. Using Mackey's theorem and Frobenius reciprocity we have,
\begin{align}
(\mbox{Ind}^G_{HG^i}\psi'|_{G^i}, 1_{G^i}) &= \sum_{x \in G^i \backslash G / HG^i}  (\mbox{Ind}_{x(HG^i)x^{-1}}^{G^i}(\psi')^x,1_{G^i})  \nonumber \\
&= \sum_{x \in G^i \backslash G / HG^i} ((\psi')^{x},1_{x(HG^i)x^{-1}\cap G^i}) \nonumber\\
&= \sum_{x \in G^i \backslash G / HG^i}(\psi',1_{HG^i}) \nonumber\\
&= |G^i \backslash G/ HG^i|.
\end{align} 
Here the last equality comes from the fact that $\psi'$ is 1-dimensional and $\psi'|_{H\cap G^i}$ is trivial. It follows that, $|G^i \backslash G /HG^i|=|G/HG^i|$, because $G^i$ is normal in $G$. But dimension of $\mbox{Ind}^G_{HG^i}(\psi')|_{G^i}$ is already $|G/HG^i|$, which simply mean $\mbox{Ind}^G_{HG^i}(\psi')|_{G^i}=c1_{G^i}$, where $c$ is given by $|G/HG^i|$.\\
\newline
Suppose $l(\chi)\leq i$, by Frobenius reciprocity it is enough to show
\[(\chi|_{HG^i}, \psi')=(\chi|_{H},\psi),\]
which one can check using the original definition of inner product of characters, and the fact that both of $\chi$ and $\psi'$ are trivial on $G^i$.
\end{proof} 
We denote the set $S^i$ of irreducible characters of $G$ of level less than or equal to $i$, and $K^i$ to be the fixed field of $G^i$. As a consequence of Lemma \ref{imp for i}, we obtain the following: 
\begin{corollary}\label{imp for i_2} Let $d$ be the greatest common divisor of dimensions of the characters in $\emph{Irr}(G)-S^i$. Then  \[\emph{ord}_{s=s_0}\Big(\zeta_K(s)/\zeta_{K^i}(s)\Big)=kd,\]
where $k$ is a non negative integer. 
\end{corollary}
\begin{proof} Taking $H=\{1\}$ and $\psi=1_H$ in Lemma \ref{imp for i}, we get
\[\mbox{Ind}_{G^i}^{G}1_{G^i}=\sum_{l(\chi)\leq i} \chi.\]
So applying $L(s,\cdot,K/F)$ to both sides and using Lemma \ref{first} we have,
\[\frac{\zeta_K(s)}{\zeta_{K^i}(s)}=\frac{\prod_{\chi \in \mbox{Irr}(G)} L(s,\chi,K/F)^{\chi(1)}}{\prod_{\chi \in S^i} L(s,\chi,K/F)^{\chi(1)}},\]
and so the result follows by comparing the order of both sides at $s=s_0$, and by the fact that $\frac{\zeta_K(s)}{\zeta_{K^i}(s)}$ is entire. (by Aramata-Brauer theorem (Theorem \ref{AB1}), proved in the next chapter.)
\end{proof}  
We now prove the most desirable result of this section.
\begin{theorem}[Lansky-Wilson, \cite{LW}] \label{LWT1} Let $K/F$ be a solvable extension of number fields, and let $G=\emph{Gal}(K/F)$. Let $\psi$ be a one-dimensional character of a subgroup $H$ of $G$. Then,
\[\frac{L(s,\emph{Ind}^G_H \psi,K/F)}{\prod_{\chi \in \emph{Irr}(G),l(\chi)\leq i} L(s,\chi,K/F)^{(\chi,\emph{Ind}_{H}^G(\psi))}}\] 
is entire.  
\end{theorem} 
\begin{proof} First we assume that $\psi$ is trivial on $H \cap G^i$. So clearly $\psi$ extends uniquely to a character $\psi'$ of $HG^i$ that is trivial on $G^i$. Lemma \ref{imp for i} implies that the denominator in the quotient is same as $L(s,\mbox{Ind}_{HG^i}^G(\psi'),K/F)$. So we only need to investigate $\mbox{Ind}^G_H(\psi)-\mbox{Ind}_{HG^i}^{G^i}(\psi').$ Following proof of Theorem 5.3 in \cite{MR}, we can write \[\mbox{Ind}^G_H(\psi)=\mbox{Ind}_{HG^i}^{G}(\psi')+\sum_{j}\mbox{Ind}_{H_j}^{G}(\psi_j \otimes \psi'|_{H_j}),\]
and following the same argument given in proof of Theorem \ref{RR2}, we see $\sum_{j}\mbox{Ind}_{H_j}^{G}(\psi_j \otimes \psi'|_{H_j})$ does not contain $1_G$. So we can now use Theorem \ref{ODT} without any problem. Hard part of the proof is when $\psi$ is not trivial on $H \cap G^i$. For an element $x\in G,$ we denote the character $\psi^x$ by $\psi^x(g)=\psi(x^{-1}gx)$. Using Mackey's theorem and Frobenius reciprocity, 
\begin{align} 
(\mbox{Ind}^G_H(\psi)|_{G^i}, 1_{G^i}) &= \sum_{x \in G^i \backslash G / H}  (\mbox{Ind}_{xHx^{-1}}^{G^i}(\psi^x),1_{G^i})  \nonumber \\
&= \sum_{x \in G^i \backslash G / H} (\psi^{x},1_{xHx^{-1}\cap G^i}) \nonumber\\
&= \sum_{x \in G^i \backslash G / H}(\psi,1_H) \nonumber\\
&= 0. 
\end{align}
as $\psi$ is not trivial on $H\cap G^i$. And so $\mbox{Ind}^G_H(\psi)$ can not contain any character of level at most $i$, because all characters with level at most $i$ restricts to a non zero multiple of $1_{G^i}$ on $G^i$. And that means $(\chi,\mbox{Ind}^G_H(\psi))=0$ for any irreducible character $\chi$ with level at most $i$. So the problem is reduced to show that $L(s,\mbox{Ind}_{H}^{G}(\psi),K/F)$ is entire, which is of course the case as per Theorem \ref{ODT}, unless $\psi=1_H$. But $\psi=1_H$ is not possible, because in this case we assumed $\psi$ is nontrivial on $H\cap G^i$. 
\end{proof} 
\begin{corollary} Let $\psi_0$ be a one-dimensional character of a subgroup $H$ of $G$, and let $S$ be the set of all irreducible characters of level $i$ occurring in $\emph{Ind}_H^G(\psi_0)$. Then, the product of L-functions
\[\prod_{\chi \in S} L(s,\chi, K/F)^{\big(\chi, \emph{Ind}_{H}^G(\psi_0)\big)}\]
is entire 
\end{corollary} 
\begin{proof} If $\psi_0$ is nontrivial on $H \cap G^i$, then by the same argument as in previous proof, we can say $(\psi,\mbox{Ind}_{H}^G(\psi_0))=0$ for any $x\in S$, and so does the proof follows. So we assume $\psi_0$ is trivial on $H \cap G^i$, and that means $\psi_0$ extends to a character $\psi$ on $HG^i$. Theorem \ref{LWT1} implies that, 
\[\frac{L(s,\mbox{Ind}^G_{HG^i}\psi,K/F)}{\prod_{\chi \in S^{i-1}} L(s,\chi,K/F)^{\big(\chi, \mbox{Ind}^G_{HG^i} (\psi)\big)}}\]
is entire. Following proof of Lemma \ref{imp for i}, above quotient is same as
\[\frac{\prod_{\chi \in S^i}L(s,\chi,F/K)^{\big(\chi,\mbox{Ind}^G_H(\psi_0)\big)}}{\prod_{\chi \in S^{i-1}}L(s,\chi,K/F)^{\big(\chi,\mbox{Ind}^G_H(\psi_0)\big)}},\]
which is same as 
\[\prod_{\chi \in S} L(s,\chi,K/F)^{\big(\chi,\mbox{Ind}^G_H(\psi_0)\big)},\]
as desired. 
\end{proof}

\section{Heilbronn characters on Artin L-functions}

To study Artin L-functions, Heilbronn introduced an innovative method. We are now going to discuss about this. 
\begin{definition} \label{HC} Let $K/F$ be a finite Galois extension of number fields with Galois group $G$. Let $s_0$ in $\mathbb{C}$ be fixed. For each subgroup $H$ of $G$ the Heilbronn character $\Theta_H$ is defined as 
$$\Theta_H(g)=\sum_{\chi}n(H,\chi)\chi(g),$$
where the summation is over all irreducible characters $\chi$ of $H$ and, 
\[n(H, \chi)=\mbox{ord}_{s=s_0}L(s,\chi, K/K^{H}).\]
\end{definition}
\noindent Definition above makes sense because Artin L-functions are known to have a meromorphic continuation (Theorem \ref{mero}). So the integers $n(H,\chi)$'s make sense. Note that Heilbronn characters are defined on all subgroups, whether they are proper or not. So there must be some kind of compatibility between characters for groups and their subgroups. Indeed it does, as can be seen below.
\begin{theorem}[Heilbronn-Stark Lemma] \label{HSL}
For any subgroup $H$ of G,
\[\Theta_G|_{H}=\Theta_H.\] 
\end{theorem}
\begin{proof} See \cite{RM} pp. 154.
\end{proof}
\noindent Here is another crucial result, 
\begin{theorem} \label{bound}
$$\Theta_G(1)=\emph{ord}_{s=s_0} \zeta_K(s).$$
\end{theorem}
\begin{proof}
Clearly, 
$$\Theta_G(1)=\sum_{\chi} n(G,\chi)\chi(1),$$
and from Artin-Takagi decomposition (see proof of Theorem \ref{AC to DC}) we can write
\[\zeta_{K}(s)=\prod_{\chi \in \mbox{Irr}(G)}L(s,\chi,K/F)^{\chi(1)}.\]
So, $\mbox{ord}_{s=a}\zeta_K(s)=\sum_{\chi} n(G,\chi)\chi(1)=\Theta_G(1).$
\end{proof}
\noindent Following result allows us to bound all order of zeroes. This is one of the keys to deal with problems related to quotient of various L-functions.
\begin{theorem}[Foote-Murty, \cite{FM}] \label{MRI}
\[\sum_{\chi}n(G, \chi)^2 \leq ( \emph{ord}_{s=s_0} (\zeta_K(s)))^2.\]
\end{theorem} \label{5} 
\begin{proof} Using orthogonality relations of characters we have, 
$$(\Theta_G,\Theta_G)=\sum_{\chi}n(G, \chi)^2,$$
and also from definition it follows,
$$(\Theta_G,\Theta_G)=\frac{1}{|G|}\sum_{g\in G}|\Theta_G(g)|^2.$$
From Heilbronn-Stark Lemma (Theorem \ref{HSL}) we know  $\Theta_G|_{\langle g \rangle}=\Theta_{\langle g \rangle}$. And so our desired summation boils down to, $\frac{1}{|G|}\sum_{g\in G}|\Theta_{\langle g \rangle}(g)|^2$. First we let $s_0 \neq 1$, i.e. we are looking at order of Artin L-functions at $s=s_0\neq 1$. We know all irreducible representations of abelian groups are one dimensional, so by Theorem \ref{ODT} all of $L\big(s,\chi,K/K^{\langle g \rangle}\big)$'s are entire, and hence $n(\langle g \rangle,\chi) \geq 0$ for any irreducible character $\chi$. This implies,
$$|\Theta_{\langle g \rangle}(g)| \leq \sum_{\chi} n(\langle g \rangle,\chi)=\Theta_{\langle g \rangle}(1)=\Theta_{G}(1),$$
and the later one is bounded by $(\mbox{ord}_{s=s_0}(\zeta_K(s)))$ (from Theorem \ref{bound}). And so,
\[(\Theta_G,\Theta_G) \leq ( \mbox{ord}_{s=a} (\zeta_K(s)))^2.\]
If $s_0=1$, from Theorem \ref{ODT}, all of $n(\langle g \rangle, \chi)$'s are zero unless $\chi$ is trivial. And when $\chi$ is trivial, this number is $-1$. In other words, $\Theta_{\langle g \rangle}(g)=-1$ for any $g\in G$. Therefore,
\[(\Theta_G,\Theta_G) =1 \leq ( \mbox{ord}_{s=1} (\zeta_K(s)))^2,\]
as $\zeta_K(s)$ has a simple pole at $s=1$.
\end{proof}
\begin{definition}[Dirichlet Density] Let $F$ be a number field. We say that a set of prime ideals $S$ of $F$ has Dirichlet density $D(S)$ if,
$$\lim_{s \to 1^{+}} \frac { \sum_{\mathfrak{p}\in S} \frac{1}{N(\mathfrak{p})^s}}{\log \zeta_F(s)} =D(S).$$
\end{definition}
\begin{theorem}[Chebotarev's Density theorem, \cite{RM}] \label{CDT1} Let $K/F$ be a finite Galois extension of algebraic number fields with Galois group $G$. If $C$ is a conjugacy class of $G$, the prime ideals $\mathfrak{p}$ of $O_F$ with $\sigma_{\mathfrak{p}} \in C$ has Dirichlet density $\frac{|C|}{|G|}$.
\end{theorem}
Chebotarev's density theorem is very important in number theory. It is a very deep theorem and comes from class field theory. We will show how Heilbronn characters serves as a tool to set up the bridge. For some of its beautiful applications in number theory, see \cite{Ser}. Let us now give a rough sketch of poof of Chebotarev's density theorem.\\
\newline
\textit{Proof of \mbox{Theorem 4.6}.} \label{CDT} From Theorem \ref{MRI} it follows that, for a non trivial character $\chi, L(s,\chi,K/F)$ has neither a zero nor a pole at $s=1$. Now we go back to discussion in Section \ref{DS}, from Dirichlet expansion of Artin L-functions we can easily deduce the following
\[\sum_{m,\sigma_{\mathfrak{p}^m} \in C}\frac{1}{N(\mathfrak{p})^{ms}}=\frac{|C|}{|G|}\sum_{\chi \in \mbox{Irr}(G)} \chi(\overline{g_C})\log L_{\mbox{unr}}(s,\chi,K/F),\] \label{ded}
\noindent for any conjugacy class $C \subset \mbox{Gal}(K/F)$ and any element $g_C \in C$. The summation in the left hand side runs over all unramified primes of $\mathfrak{p}$ in $O_K$. By Theorem \ref{MRI} we can say $L(1,\chi,K/F)\neq 0$ for any nontrivial irreducible character $\chi$ of $\mbox{Gal}(K/F)$. As ramified factor of $L(s,\chi,K/F)$ is a finite product of functions, converging in $\mbox{Re}(s)\geq 1$, so 
\[L(1,\chi,K/F) \neq 0 \implies L_{ram}(1,\chi,K/F) \neq 0,\]
for any irreducible and nontrivial $\chi$. Now dividing both sides of \ref{ded} by $\log(\zeta_F(s))$, and using the fact that $\log(L(s,\chi,K/F))=\log(L_{unr}(s,\chi,K/F))$+finite part in $\mbox{Re}(s)\geq 1$, we get the desired result. See \cite{RM} for a more detailed explanation.
\subsection{Further variants of Murty-Raghuram's inequality and applications}
We start this section with proving Dedekind's conjecture when $L/K$ is Galois. We previously proved this under Artin's conjecture at Theorem \ref{AC to DC}. Now let us prove this unconditionally. This was first proved by Brauer in \cite{BAM}. We explain Foote-Murty's \cite{FM} proof below.
\begin{theorem}[Aramata-Brauer] \label{AB1} $\zeta_K(s)/\zeta_F(s)$ is entire when $K/F$ is a finite Galois extension of number fields.
\end{theorem} 
\begin{proof} The proof is an immediate corollary of Theorem \ref{MRI}. Because we have, 
$$|n(G,\chi)|^2 \leq \mbox{ord}_{s=a}(\zeta_K(s))^2,$$
for any irreducible character $\chi$. Now, on taking $\chi=1_G$, we get 
$$|\mbox{ord}_{s=a}\zeta_F(s)|^2 \leq |\mbox{ord}_{s=a}\zeta_K(s)|^2.$$
We know for any $s_0 \neq 1$, Dedekind zeta functions are holomorphic at $s=s_0$ (by Theorem \ref{ODT}), and so both of $\mbox{ord}_{s=a}\zeta_F(s)$ and $\mbox{ord}_{s=a}\zeta_K(s)$ are non negative. So the theorem is proved for $s_0 \neq 1$. But at $s_0=1$, any Dedekind zeta function has a pole of order one. That pole gets canceled when we take quotient of them.
\end{proof} 
\noindent The significance of this formalism will be clear from the following result.
\begin{theorem}[Murty-Raghuram, \cite{MR}] \label{MR AB} 
Let $G$ be the Galois group of a solvable Galois extension $K/F$ of number fields. Let $K_1$ be the fixed field of $[G,G]=G^1$. Then the quotient 
\[\frac{\zeta_K(s)}{\zeta_{K_1}(s)}\]
is entire. Moreover it does not have any simple zero.
\end{theorem}  
We will first prove the following lemma, a variant of Theorem \ref{MRI}.
\begin{lemma}\label{MRI 2}Let $G$ be the Galois group of a solvable Galois extension $K/F$ of number fields. Let $K_1$ be the fixed field of $[G,G]=G^1$ and $S$ be set of all one dimensional characters of $\chi$. Then
\[\sum_{\chi \not\in S} n(G,\chi)^2 \leq \emph{ord}_{s=s_0}\Big(\frac{\zeta_K(s)}{\zeta_{K_1}(s)}\Big)^2.\]
\end{lemma}
\begin{proof} Let us write, 
\[\Theta^{S}_{G}=\sum_{\chi \not\in S} n(G,\chi)\chi.\]
Therefore,
\[\frac{1}{|G|} \sum_{g\in G}|\Theta^S_{G}(g)|^2 = \sum_{\chi \not\in S}n(G, \chi)^2.\]
On the other hand, 
\begin{align} 
\Theta^S_G(g) & = \Theta_G(g)-\sum_{\chi \in S} n(G, \chi)\chi(g)\\ 
& =\sum_{\phi \in \mbox{Irr}(\langle g \rangle)}\Big(n(\langle g \rangle, \phi)-\sum_{\chi \in S} n(G, \chi)(\chi,\mbox{Ind}_{\langle g \rangle}^{G}\phi)\Big) \phi(g).
\end{align} 
We obtained (4.2) from (4.1) using Frobenius reciprocity. Theorem \ref{RR2} says for any $\phi$ in $\mbox{Irr}(\langle g \rangle)$,
\[n(\langle g \rangle, \phi)-\sum_{\chi \in S_{\phi}} n(G, \chi) \geq 0.\]\label{line}
\noindent But by Frobenius reciprocity, $\chi \in S_{\phi}$ if and only if $(\chi,\mbox{Ind}_{\langle g \rangle}^{G} \phi)=1$ and $\chi \not\in S_{\phi}$ if and only if $(\chi,\mbox{Ind}_{\langle g \rangle}^{G} \phi)=0$. Which means from line 4.2 it follows,
\[n(\langle g \rangle, \phi)-\sum_{\chi \in S} n(G, \chi)(\chi,\mbox{Ind}_{\langle g \rangle}^G \phi) \geq 0.\] 
\noindent But the interesting part is 
\[\sum_{\chi \in S} \chi = \mbox{Ind}_{[G,G]}^{G}(1_{[G,G]}).\]
This follows by Frobenius reciprocity and the fact that number of all one dimensional characters of $G$ are $|G/[G,G]|$. In other words, commutator subgroup is the best possible object to understand this one dimensional characters. And so,
\begin{align} 
|\Theta^S_G(g)| & \leq \sum_{\psi \in \mbox{Irr}(\langle g \rangle)} \Big(n(\langle g \rangle, \phi)-\sum_{\chi \in S} n(G, \chi)(\chi,\mbox{Ind}_{\langle g \rangle}^G \phi)\Big)\\
& = \mbox{ord}_{s=s_0}\Big(\frac{\prod_{\psi \in \mbox{Irr}(\langle g \rangle)} L(s,\psi, K/K^{\langle g \rangle})}{\prod_{\psi \in \mbox{Irr}(\langle g \rangle)} \prod_{\chi \in S} L(s,\chi,K/F)^{(\chi,\mbox{Ind}_{\langle g \rangle}^{G}\psi)}}\Big)\\
& = \mbox{ord}_{s=s_0} \Big(\frac{\zeta_K(s)}{\prod_{\chi \in S} L(s,\chi,K/F)}\Big)\\
& = \mbox{ord}_{s=s_0}\Big(\frac{\zeta_K(s)}{L(s,\mbox{Ind}_{[G,G]}^G(1_{[G,G]}),K/F)}\Big)\\
& = \mbox{ord}_{s=s_0} \Big(\frac{\zeta_K(s)}{\zeta_{K^{1}}(s)}\Big), 
\end{align}
as desired. Here denominator of line 4.5 is obtained from line 4.4, using Frobenius reciprocity and numerator is obtained using the fact that $\Theta_G(1)=\Theta_{\langle g \rangle}(1)$.
\end{proof}
\textit {Proof of Theorem \ref{MR AB}}. As $K/K^1$ is a Galois extension, so by Aramata-Brauer theorem we see the quotient $\frac{\zeta_K(s)}{\zeta_{K_1}(s)}$ has no poles. But if the quotient has a simple zero at some point $s=s_0$, then from Lemma \ref{MRI 2} it follows that $|n(G,\chi)|=1$ for at most one $\chi \in S$, and others are zero. We have 
\[\mbox{Reg}_G=\mbox{Ind}_{[G,G]}^{G} 1_{[G,G]}+\sum_{x \not\in S}\chi(1) \chi,\]
and this shows
\[\frac{\zeta_K(s)}{\zeta_{K^{1}}(s)}= \prod_{\chi \not\in S}L(s,\chi)^{\chi(1)},\]
but we know left hand side has a simple zero at $s_0$, while on the right hand side, $|n(G,\chi)|=1$ for at most one $\chi \in S$, say for $\chi=\chi_0$ and others are zero. This implies $\chi_0(1)=1$, contradiction as $\chi_0 \not\in S$. \\
\newline
\noindent Following similar procedure, i.e. truncating $\Theta_G$ at a fixed one dimensional character $\chi_0 \in S$, we can deduce the following result
\begin{theorem}[Murty-Raghuram, \cite{MR}] \label{M-R I}Let $G$ be the Galois group of a solvable Galois extension $K/F$ of number fields. Let $\chi_0 \in S$, then
\[\sum_{\chi \neq \chi_0,\chi \in \emph{Irr}(G)} n(G,\chi)^2 \leq \emph{ord}_{s=s_0}  \Big(\frac{\zeta_K(s)}{L(s,\chi,K/F)}\Big)^2.\]
\end{theorem}
\subsection{Further generalization to higher levels} We now state and prove a further refinement of Lemma \ref{MRI 2}. This was done by Lansky-Wilson in \cite{LW}.
\begin{theorem} \label{LWi} Let $K/F$ be a solvable extension of number fields, and let $G=\emph{Gal}(K/F)$. Then for any $i \geq 1$,
\[\sum_{\chi \in \emph{Irr}(G),l(\chi)=i}n(G,\chi)^2 \leq \Big(\emph{ord}_{s=s_0}\Big(\frac{\zeta_{K^i}(s)}{\zeta_{K^{i-1}}(s)}\Big)\Big)^2.\]
\end{theorem}
\begin{proof} We consider the truncated series
\[\Theta^{i}_G=\sum_{\chi \in \mbox{Irr}(G),l(\chi)=i} n(G,\chi)\chi.\]
On the other hand, 
\[|\Theta^i_G|\leq \sum_{\chi \in \mbox{Irr}(G),l(\chi)=i} \chi(1)n(G,\chi)=\mbox{ord}_{s=s_0}\Big(\prod_{\chi \in \mbox{Irr}(G),l(\chi)=i}L(s,\chi,K/F)^{\chi(1)}\Big).\]
From Lemma \ref{imp for i} or proof of Corollary \ref{imp for i_2}, we have 
\[\prod_{\chi \not\in S^i}L(s,\chi,K/F)^{\chi(1)}=\frac{\zeta_K(s)}{\zeta_{K^i}(s)},\]
which means
\[\frac{\mbox{ord}_{s=s_0}\Big(\prod_{\chi \not\in S^i}L(s,\chi,K/F)^{\chi(1)}\Big)}{\mbox{ord}_{s=s_0}\Big(\prod_{\chi \not\in S^{i-1}}L(s,\chi,K/F)^{\chi(1)}\Big)}=\mbox{ord}_{s=s_0}\Big(\frac{\zeta_{K^i}(s)}{\zeta_{K^{i-1}}(s)}\Big),\]
and left hand side above is same as
\[\mbox{ord}_{s=s_0}\Big(\prod_{l(\chi)=i}L(s,\chi, K/F)\Big),\]
and so the proof is finally complete.
\end{proof}
\begin{rem} \label{rem5} Interestingly one can see that the result above is not really analogous to Lemma \ref{MRI 2}. For $i=1$, Lemma \ref{MRI 2} is stronger than this. Later in 2017, Wong improved it. We will talk about this in the next chapter.
\end{rem}
We finish this section with two immediate corollaries below. 
\begin{corollary} If 
\[\emph{ord}_{s=s_0}(\zeta_{K^i}(s))= \emph{ord}_{s=s_0}(\zeta_{K^{i-1}}(s)),\]
then $n(G,\chi)=0$ for every irreducible character $\chi$ of $G$ of level $i$.
\end{corollary}
\begin{proof} From the given condition we have,
\[\sum_{l(\chi)=i} n(G,\chi)^2 \leq 0,\]
and so the proof follows.
\end{proof} 
\begin{corollary} \label{crr} $\zeta_{K^i}(s)/\zeta_{K^{i-1}}(s)$ cannot have any poles or simple zeros.
\end{corollary}
\begin{proof} Proof is immediate from Theorem \ref{LWi}. We can also prove the corollary by taking $G=G^i$ in Theorem \ref{MR AB}. 
\end{proof}
\subsection{Stark's theorem and further generalization} 
\begin{theorem}[Stark, \cite{Stark}] \label{Stark} If $\emph{ord}_{s=s_0}(\zeta_K(s))\leq 1$ then for any character $\chi$ of the Galois group $G$ of $K$ over $F,L(s, \chi,K/F)$ is analytic at the point $s=s_0$.
\end{theorem}
\begin{proof} In this case $\mbox{ord}_{s=s_0}(\zeta_K(s))$ can be one of $0,1$ or $-1$. So whatever $s_0 \in \mathbb{C}$ is, we always have 
\[ \sum n(G, \chi)^2 \leq 1,\]
from Theorem \ref{MRI}. By meromorphy of Artin L-series, we have that each of $n(G, \chi)$ is an integer. The inequality from Theorem \ref{MRI} implies, for at most one $\chi$, we have $|n(G, \chi)| =1$. If $\chi$ is not one dimensional, then Artin-Takagi decomposition (see proof of Theorem \ref{AC to DC}) of $\zeta_K(s)$ gives a contradiction for the corresponding L-function, introducing a pole or zero of order greater than $1$. Hence $\chi$ is one-dimensional, but in this case the proof follows from Theorem \ref{ODT}, for $s_0 \neq 1$. And for $s_0=1$, the proof follows from the fact that $L(s,\chi,K/F)$ does not have any pole or zero at $s=1$, see proof of Theorem \ref{CDT1} for instance. 
\end{proof} 
\noindent The proof of Stark's result above show that the set of zeros and poles of all Artin L-series $L(s, \chi, K/F)$ is contained in the set of zeros of the Dedekind zeta function $\zeta_K(s)$ for all characters $\chi$ of $G$. Moreover, the order of zero of the zeta function $\zeta_K(s)$ is a critical parameter in restricting the possibility that some L-function may have a pole (or zero) at $s=s_0$. In fact, the argument tells us somewhat more. Not only is $\Theta_G$ a character (because all Artin L-functions are holomorphic), it must be a one-dimensional character (because $\Theta_G(1_G)=\mbox{ord}_{s=s_0}(\zeta_K(s))=1)$. Moreover, if $s_0$ is real, then $\chi$ must be real-valued because, if $L(s_0, \chi, K/F) = 0$, then $L(\overline{s_0}, \overline{\chi}, K/F)=0$ also, and both $\chi$ and $\overline{\chi}$ occur as constituents of $\Theta_G$, hence $\chi = \overline{\chi}$. But then its kernel is a subgroup $H$ of index at most $2$ in $G$. For any subgroup $H_1$ of $H$, the Heilbronn character $\Theta_{H_1} = \Theta_G|_{H_1}$ is a character of $H_1$. This says if $K_1$ is the fixed field of $H_1$, then the Dedekind zeta function of $E_1$ has a simple zero at $s=s_0$ (By Brauer Aramata theorem \ref{BAT}) and all $L(s, \psi, K/K_1)$ are analytic at $s=s_0$ for all irreducible characters $\psi$ of $H_1$. As a consequence, we have the following result of Stark.
\begin{theorem}[Stark, \cite{Stark}]If the Dedekind zeta function of $K$ has a real simple zero at $s=s_0$, then there is a subfield $E$ of $K$ containing $F$ and of degree at most $2$ over $F$ such that for any intermediate field $E \subseteq K_1 \subseteq K$, the Dedekind zeta function of $K_1$ has a simple zero at $s=s_0$. 
\end{theorem}


But there may exist infinitely many $s_0$ with $\mbox{ord}_{s=s_0} (\zeta_K(s))>1$. Foote and Murty \cite{FM} established a partial generalization to Stark's theorem for solvable Galois extensions. They proved,  
\begin{theorem}\label{foo} For solvable $G$ with $|G|=p_1^{n_1}p_2^{n_2} \cdots p_k^{n_k}$ where $p_1 <p_2< \cdots < p_k$ are distinct 
primes, if $\emph{ord}_{s=s_0} (\zeta_K(s)) \leq p_2-2$, then all Artin L-functions $L(s, \chi, K/F)$ are holomorphic at $s = s_0$. 
\end{theorem}  
One can naturally ask for a possible generalization. This we do by proving the following theorem in next section. This result due to recent work by P.J Wong in \cite{PJW}. 
\begin{theorem}[Wong, \cite{PJW}] \label{Won} For solvable group $G$ with $|G|=p_1^{n_1}p_2^{n_2} \cdots p_k^{n_k}$ where $p_1 < \cdots < p_k$ are distinct
primes, if $\emph{ord}_{s=s_0} \Big(\zeta_K(s)/\zeta_F(s)\Big) \leq p_2-2$, then all Artin L-functions $L(s, \chi, K/F)$s are holomorphic at $s = s_0$.
\end{theorem} 
Before going into the proof of Theorem \ref{Won} in a very orthodox manner, let us first set our playing ground. 
\begin{definition} Let G be any finite group associated to a Galois extension of number field. We say G has the property P, whenever H is either a subgroup or a quotient group of G, then H also satisfies P. We also assume solvability is included in property P.
\end{definition} 
\noindent Let $S$ be the collection of isomorphic classes of finite groups with the property P, and let $\{r_G\}_{G\in S}$ be a set of integers such that if $H$ is either a subgroup or a quotient group of $G$, then $r_G \leq r_H$. Let $S'$ be the collection consisting of groups $G \in S$ such that all Artin L-functions $L(s, \chi, K/F)$ are holomorphic at $s = s_0$ if $\mbox{ord}_{s=s_0}\Big(\zeta_K(s)/\zeta_F(s)\Big) \leq r_G$. Pick a minimal element $G_{mc}$ from $S'$. (We are taking minimal element depending on its order) Note that choice of this $G_{mc}$ depends on the $\{r_G\}_{G \in S}$. Without loss of generality we may assume $G=G_{mc}$. Now we are all set to proceed to the proof of Theorem \ref{Won}. We shall first prove some lemmas which will be required for proof of the theorem.
\begin{lemma} \label{RW1}$\Theta_G$ is not a character of $G$, but $\Theta_G|_H$ is a character of any proper subgroup $H$ of $G$. \end{lemma}
\begin{proof} We assumed $G=G_{mc}$, so from definition it is clear that there is an irreducible character $\chi$ of $G$ such that $n_\chi$ is negative. But $n_\chi$ is the coefficient of $\chi$ of Heilbronn character $\Theta_G$, showing first part of the lemma.\\ 
\newline
\noindent If $H$ is a subgroup of $G$, we have
\[\mbox{ord}_{s=s_0}\Big(\zeta_K(s)/\zeta_F(s)\Big) \leq r_G \leq r_H.\]
As P contains the property of solvability, $G$ itself is a solvable group. So by Uchida-van der Waall theorem (Theorem \ref{UVdW}), we see $\zeta_M(s)/\zeta_F(s)$ is entire, where $M$ is the fixed field of $H$. As, $K$ is normal closure of $M$ over $F$ and $\mbox{Gal}(K/F)$ is solvable. So,  
\[\mbox{ord}_{s=s_0}\Big(\zeta_K(s)/\zeta_M(s)\Big) \leq \mbox{ord}_{s=s_0}\Big(\zeta_K(s)/\zeta_F(s)\Big) \leq r_H,\]
and later inequality suggests $H$ is a potential member of $S'$. But as $H$ is a proper subgroup of $G$, by minimality of $G=G_{mc}$ we see $n_\phi$ is non-negative for any irreducible character $\phi$ of $H$, proving the lemma. 
\end{proof}
\begin{lemma} \label{RW2} If $\chi$ is an irreducible character of $G$ such that $n_{\chi} < 0$, then $\chi$ is faithful. In other words $\emph{Ker} (\chi)$ is trivial.
\end{lemma}
\begin{proof} Suppose $\text{Ker} (\chi)$  is non-trivial and let $M$ denote its fixed field. Since $\text{Ker} (\chi)$ is a normal subgroup of G, we can consider $\chi$ to be a character $\chi_0$ of $\mbox{Gal}(M/F)$, whose inflation to $\mbox{Gal}(K/F)$ is same as $\chi$. By Lemma \ref{first} $L(s, \chi, K/F)$ can be viewed as an Artin L-function for $M/F$ attached to the character $\chi_0$. We denote $N=\mbox{Gal}(M/F)$. By Aramata-Brauer theorem, $\Big(\zeta_K(s)/\zeta_F(s)\Big)$ and $\Big(\zeta_K(s)/\zeta_M(s)\Big)$ are entire, so
\[\mbox{ord}_{s=s_0}\Big(\zeta_M(s)/\zeta_F(s)\Big) \leq r_G \leq r_N.\]
Again as in the previous proof, above expression suggests that $N$ is potential member of $S'$, but minimality of $G=G_{mc}$ forces $n_{\chi_0} \geq 0$. And by inflation-invariant property (Lemma \ref{first}) of Artin L-functions, $n_{\chi_0}=n_{\chi} \geq 0$, a contradiction.
\end{proof}
\begin{lemma} \label{RW3} If $\chi$ is an irreducible character of $G$ such that $n_{\chi} < 0$, then $\chi$ is not induced from any proper subgroup of $G.$
\end{lemma}
\begin{proof} For sake of contradiction, suppose there exist a character $\psi$ for a proper subgroup $H$ of $G$ such that $\chi = \mbox{Ind}^G_H \psi$, then we have,
\[n_{\chi} = \mbox{ord}_{s=s_0} L(s, \psi, K/K^H),\]
but the latter one is non-negative by Lemma \ref{RW1}, a contradiction. 
\end{proof}
\noindent Now let us decompose $\Theta_G$ into three parts. Denote $\Theta_3$ to be the sum of all terms $n_{\chi}\chi$ so that $\chi$ is not faithful, and let $-\Theta_2$ be the sum of all terms $n_{\chi}\chi$ for which $n_{\chi}$ is negative. Lemma \ref{RW2} implies that $(\Theta_2,\Theta_3)=0$. We further define $\Theta_1=\Theta_{G}+\Theta_2-\Theta_3$
\begin{claim} Since $G$ is solvable and non-abelian, it admits a normal subgroup $N$ of prime index $p$ (say) that contains the centre $Z(G)$ of $G$.
\end{claim}
\begin{proof} As $G$ is non-abelian, $G/Z(G)$ is non trivial whose center is trivial. So it is enough to show existence of a normal subgroup $N$ having prime index $p$, for some prime $p$. Let $A$ be a (non trivial) minimal normal subgroup of $G$. From \ref{yea} we know $A$ is abelian. Let $p\mid |A|$, then $A$ has a $p$-Sylow subgroup $S$. Since $A$ is normal in $G$ and $S$ is characteristic in $A$ (because it is the unique Sylow subgroup since $A$ is abelian) we must have that $S$ is normal in $G$, by assumption this implies that $A=S$ so that $|A|=p^n$ for some $n$. If $G$ is not a $p$-group, $G/A$ is non trivial having order less than $|G|$. Then by induction hypothesis, we get a required normal subgroup in $G/A$, and so in $G$. So the obstacle only comes when $G=A$, i.e. $G$ is a $p$-group. But in that case it is very well known that, for any divisor $d$ of $|G|$, there exist a subgroup of index $d$. We are finally done taking $d=p$ and using the fact that any subgroup of index $p$ is normal, (because $p$ is the smallest prime divisor of $|G|$). 
\end{proof} 
\begin{lemma} \label{RW4} Let $N$ be a normal subgroup of $G$ of prime index $p$ containing the centre $Z(G)$ of $G$. Then we have,
\[(\Theta_2|_N , \Theta_3|_N ) = 0.\]
Further, $\Theta_1|_N-\Theta_2|_N$ is either zero or is equal to a character $\phi$ of $N.$
\end{lemma} 
\begin{proof} As $[G : N]$ is prime, by Theorem \ref{Cli} for every $\chi \in \mbox{Irr}(G)$, either $\chi|_N$ is irreducible or $\chi$ is induced from $N$. So if $\chi$ is contained in $\Theta_2$, then $\chi|_{N}$ is of course irreducible (from Lemma \ref{RW3}). The theorem by Blichfeldt (Theorem \ref{Bli}) says, if $G$ has a abelian normal subgroup not contained in center of $G$, then every faithful irreducible character of $G$ would be induced from a proper subgroup, contradiction to Lemma \ref{RW3}. So all abelian normal subgroups of $G$ must be contained in the centre $Z(G)$ of $G$. For every non-faithful $\chi \in \mbox{Irr}(G)$, we know $\mbox{Ker}(\chi)$ is a non trivial normal subgroup of $G$. By solvability of $G$ we can get a (non trivial) abelian normal subgroup $H$ of $G$ contained in $\ker(\chi)$. But we just observed earlier that $H$ is contained in $Z(G)$. So, we can say \[\ker(\chi) \cap Z(G)\] 
is non empty. So none of the irreducible components of $\Theta_3$ is faithful on $Z(G)$. But as per assumption $Z(G) \subseteq N$, so set of all irreducible components of $\Theta_2|_N$ and $\Theta_3|_N$ are pairwise disjoint, and so first part of the lemma is established.\\ 
\newline
For second part, from Lemma \ref{RW1} we know $\Theta_N=\Theta_1|_N-\Theta_2|_N+\Theta_3|_N$ is a character of $N$. If $\Theta_1|_N-\Theta_2|_N$ is identically zero then we are done. Otherwise, suppose it is not a character, so there exist an irreducible character $\psi$ of $N$, which comes with negative coefficient in $\Theta_1|_N-\Theta_2|_N$. But since $\Theta_N$ is a character, $\psi$ must come with positive coefficient in $\Theta_3|_{N}$. From definition of $\Theta_1$, if $\psi$ comes as a component in $\Theta_1|_N$, (here by component we mean, $\psi$ comes as a term, when we write $\Theta_1|_N$ as sum of irreducible characters of $N$) it would come with a positive coefficient, which means it comes with a (strictly) negative coefficient in $\Theta_2|_{N}$. Contradiction as  $(\Theta_2|_{N},\Theta_3|_{N})=0$. So, $\Theta_1|_N-\Theta_2|_N$ is a character of $N$.
\end{proof} 
\begin{lemma} \label{RW5} Suppose that $\Theta_1|_N-\Theta_2|_N$ is a character $\phi$ of $N$. Then there is an irreducible
character $\psi$ appearing in $\Theta_1 - \Theta_2$ such that $\psi(1) \leq r_G$.
\end{lemma}
\begin{proof} Let $\phi_1$ be an irreducible component of $\Theta_1|_N-\Theta_2|_N$. Pick an irreducible component $\phi$ of $\Theta_1-\Theta_2$ such that $\phi|_N$ contains $\phi_1$. If they are equal then, taking $\psi=\phi$ we have, 
\[\psi(1)=\phi_1(1) \leq (\Theta_1|_N-\Theta_2|_N)(1).\]
If not, then as $[G:N]=p$, by Theorem \ref{Cli} we can write 
\[\psi|_{N} = \sum_{i=1}^{p} \phi_i,\]
such that $\psi=\mbox{Ind}_{N}^{G} \phi_i$. But on the other hand $\psi$ being an irreducible component of $\Theta_1-\Theta_2$, we can say all of $\phi_i$'s are contained in $\Theta_1|_{N}-\Theta_2|_{N}$. And hence,
\[\psi(1)=\sum_{i=1}^{p}\phi_i(1) \leq (\Theta_1|_{N}-\Theta_2|_{N})(1).\]
So it is enough to show that, $(\Theta_1|_{N}-\Theta_2|_{N})(1) \leq r_G$. Together with our assumption and the fact that the trivial character is not faithful, we can write
\[r_{G} \geq  \mbox{ord}_{s=s_0}\Big(\zeta_K(s)/\zeta_F(s)\Big)=\Theta_{G}(1)-n_{1_G}=\Theta_{N}(1)-n_{1_G}1_{N}(1).\]
Later one is nothing but
\[\Theta_1|_{N}(1)-\Theta_2|_{N}(1)+\Theta_3|_{N}(1)-n_{1_G}1_N(1),\]
and to conclude the proof, we need to just ensure that $\Theta_3|_{N}(1)-n_{1_G}1_N(1)$ is non negative. Which of course is, because $1_N$ comes in $\Theta_3|_N$ with coefficient at least $n_{1_G}$, (because $1_G$ is a non faithful character), and coefficient of any $\psi \in \mbox{Irr}(N)$ coming in $\Theta_3|_N$ has non negative coefficient.
 \end{proof}
\begin{lemma} \label{4.26} Let $H$ be the subgroup generated by $Z(G)$ and an element $x \in G \setminus N.$ Then $(\Theta_2|_H, \Theta_3|_H ) =0$ and $\Theta_1|_H-\Theta_2|_H$ is either $0$ or a character of $H$. 
\end{lemma}
\begin{proof} Proof is just analogous to the arguments given in Lemma \ref{RW4}. By construction, $H$ is a normal abelian subgroup of $G$. We may assume $H$ to be a proper subgroup of $G$, otherwise $G$ would be abelian, and we know Artin's conjecture holds true for abelian extensions (because all irreducible characters of abelian groups are one-dimensional). From proof of Lemma \ref{RW4} we know for every irreducible component $\chi$ of $\Theta_3, \mbox{Ker} (\chi) \cap Z(G)$ is non-trivial. Take some element $g \in \mbox{Ker} (\chi) \cap Z(G)$. As $H \subseteq Z(G)$, from definition one can easily show $g$ is in kernal of $\mbox{Ind}^G_H(\chi|_H)$. As any irreducible component $\lambda$ of $\Theta_2$ is faithful, one has 
\[(\lambda, \mbox{Ind}^G_H(\chi|_H))=0,\]
and by Frobenius reciprocity $(\lambda|_{H}, \chi|_H)=0$. So first part of the lemma is proved. Again by same arguments as in the proof of Lemma \ref{RW4}, second part is also established.
\end{proof} 
\begin{lemma} \label{RW6} If the minimal counterexample $G=G_{mc}$ exists, then $\Theta_1|_N \neq \Theta_2|_N$.
\end{lemma} 
\begin{proof} Suppose that $\Theta_1|_N = \Theta_2|_N$. But of course $\Theta_1(1)=\Theta_2(1)$, and so from Lemma \ref{4.26} we have, $\Theta_1|_H = \Theta_2|_H$ for any subgroup $H$ generated by $Z(G)$ and $x\in G-N$. So, $\Theta_1(x)=\Theta_2(x)$ for any $x \in G$, as they agree on $N$. So $\Theta_G=\Theta_3$, but this a contradiction to Lemma \ref{RW1}, as $\Theta_3$ is a character of $G$, by Lemma \ref{RW2}.
\end{proof} 
\noindent \textit{Proof of \mbox{Theorem \ref{Won}}}. First of all, we let P to be the property of solvablity. For each $G$ satisfying the property $P$, we set $r_G$ to be equal to $p_2-2$, where $p_2$ denotes the second smallest prime divisor of $|G|$. For a sake of contradiction, we take $G$ as a minimal counterexample. Following the notation of Lemma \ref{RW4}, we may assume that $\Theta_1|_N-\Theta_2|_N$ is a character of N. Then Lemma \ref{RW5} asserts that there is a faithful irreducible character $\psi$ contained in $\Theta_1-\Theta_2$ such that $\psi(1)\leq p_2-2$. But $G$ is solvable, so the theorem by Ito \ref{Ito} tells us any Sylow $q-$subgroup $G$ is abelian and normal, for all $q>p_1$. So $G$ admits an abelian normal subgroup $A$ such that $G/A$ is of order $p_1^{n_1}$. However Huppert's result \ref{Hup} forces $G$ to be an $M-$group, (as any $p$-group is supersolvable) a contradiction, since Artin's conjecture \ref{ACc} is of course true for an $M-$ group. Therefore, $\Theta_1|_N=\Theta_2|_N$. Finally Lemma \ref{RW6} leads to a contradiction. \\
\newline 
\noindent We have one immediate and interesting corollary to Theorem \ref{Won}.
\begin{corollary} Assume that $K/F$ is Galois and of odd degree. If 
\[\emph{ord}_{s=s_0} \Big(\zeta_K(s)/\zeta_F(s)\Big) \leq 3,\] 
then every Artin L-function $L(s, \chi, K/F)$ is holomorphic at $s = s_0$.
\end{corollary}  
\begin{proof} This is immediate by a celebrated theorem of Feit and Thompson \cite{FT}, which says that all groups of odd order are solvable.
\end{proof} 
\noindent As a sequel one may of course ask what if $|\mbox{Gal}(K/F)|$ is even. Yes, we then have the following result.
\begin{theorem}[Wong, \cite{PJW}] \label{Won2} Let $K/F$ be a Galois extension of number fields with Galois group $G$. If $4$ does not divide $|G|$ and $\emph{ord}_{s=s_0}\Big(\zeta_{K}(s)/\zeta_k(s)\Big) \leq 3,$ then all Artin L-functions $L(s, \chi, K/F)$ are holomorphic at $s = s_0$.
\end{theorem} 
\begin{proof} First we need to prove the following lemma.
\begin{lemma} Let $K/F$ be a solvable Galois extension of number fields with Galois group $G$. If
\[|G|=q^lp^n\prod_{i=1}^{k} p_i^{n_i}\]
with $q<p<p_1<p_2<\cdots<p_k, q|p-1$ and $l\leq 1$. If $\emph{ord}_{s=s_0}\Big(\zeta_K(s)/\zeta_F(s)\Big) \leq  p_1-2$, then all Artin L-functions $L(s, \chi, K/F)$'s are holomorphic at $s=s_0$.
\end{lemma} 
\begin{proof}
First we fix the primes $q,p,p_1,p_2,\cdots,p_k$ in our hand, and let P to be the property that $G$ is a solvable group of order $q^lp^n\prod_{i=1}^{k} p_i^{n_i}$ with $q<p<p_1<p_2<\cdots<p_k, q|p-1, l\leq 1$ and $l,n,n_i$'s are all non negative integers. One can easily check our "declared" property P is indeed a property on $G$ by definition. Consider the set of such groups having property P, and set $r_G=p_1-2$ for any such groups. We further take $G=G_{mc}$ to be the minimal counterexample. Following the same proof and same notation as we did in for Theorem \ref{Won}, we may assume that $\Theta_1|_N-\Theta_2|_N$ is a character of $N$. From Lemma \ref{RW5} we can say there exist an irreducible and faithful character $\psi$ appearing in $\Theta_1-\Theta_2$ such that $\psi(1)\leq r_G=p_1-2$. Again by Ito's theorem \ref{Ito} we can say Sylow $p_j$-subgroups of $G$ are abelian and normal. So $G$ admits an abelian normal subgroup $A$ of index $qp^n$. But since $q<p$, from Sylow's theorem we can say $p$-Sylow subgroup of $G/A$ is normal. But a $p$-group is always supersolvable, so $G/A$ is supersolvable. So by Huppert's theorem \ref{Hup} once again, $G$ is a $M$-group, a contradiction.
\end{proof} 
\noindent Now as for the proof of Theorem \ref{Won2}, we just take $q=2$. Then we get $3 \leq p_1-2$, and the proof follows immediately. 
\end{proof} 
\noindent Theorem \ref{foo} was further generalized by Foote \cite{33} in 1990. Define $M_G$ to be the minimum, taken over all subgroups $H$ of $G$, of the degrees of all nonmonomial irreducible characters of $H$. Then the following result holds
\begin{theorem}[Foote, \cite{33}] \label{tt} Let $E/F$ be a Galois extension of number fields with solvable Galois group $G$, and let $s_0 \in \mathbb{C}-\{1\}$. If $\emph{ord}_{s=s_0} \zeta_K(s) < M_G$, then all Artin L-functions $L(s, \chi, K/F)$ are analytic at $s=s_0$ for every irreducible character $\chi$ of $G$.
\end{theorem}
\noindent Corollary 4 of \cite{33} shows $M_G \geq p_2-1$, so the result above indeed generalizes Theorem \ref{foo}. The proof of Theorem \ref{tt} follows the arguments as given in proof of Theorem \ref{Won}. The essential difference is, there we looked at the decomposition of $\Theta_G$ on the center of $Z(G)$, this argument considers the normal extraspecial $p$-subgroup $E$ mentioned in part (3) of Theorem 2.2 at \cite{33}., and works on the subgroup $H$ generated by $Z(E)$ and $x$, where $\Theta_1(x) \neq \Theta_2(x)$. Using the upper bound on $M_G$ provided at part (4) of Theorem 2.2, the arguments at page 270 shows $(\Theta_2|_{H},\Theta_3|_{H})=0$ and proceeds as we did after Lemma \ref{4.26}.



\section{Variant of Heilbronn characters and applications to L-functions} 
\noindent Our aim of this chapter is to introduce arithmetic Heilbronn characters and their properties analogous to results we established for Heilbronn characters. We will see this formalism has several arithmetic applications on various $L$-functions. 
\subsection {Weak arithmetic Heilbronn characters}
First we will define weak arithmetic Heilbronn characters that generalize the classical Heilbronn characters, and we will discuss several properties.
\begin{definition} \label{WAC} Suppose that there is a set of integers $\{n(H, \phi)\}_{(H,\phi)\in I(G)},$ where $I(G)$ is set of all pairs $(H, \phi)$ such that either $H$ is a cyclic subgroup of $G$ or $H = G$, and $\phi$ is a character of $H$, satisfying the following three properties: 
\begin{itemize} 
\item $n(H, \phi_1 + \phi_2) = n(H, \phi_1) + n(H, \phi_2)$ for any characters $\phi_1$ and $\phi_2$ of $H$,
where $H$ is a cyclic subgroup or an improper subgroup of $G$.
\item $n(G,\emph{Ind}^G_H \phi) = n(H, \phi)$ for every cyclic subgroup $H$ and every character $\phi$ of $H$.
\item $n(H, \phi) \geq 0$ for all cyclic subgroups $H$ of $G$ and all characters $\phi$ of $H$.
\end{itemize} 
\noindent Then we define weak arithmetic Heilbronn character of a subgroup $H$ of $G$, which is either cyclic or $G$ itself, by
\[\Theta_H = \sum_{\phi \in \mbox{Irr}(H)} n(H, \phi)\phi.\]
\end{definition} 
\noindent Now we will see how they are matched up with Heilbronn characters in terms of Arithmetic properties. The first one of them is the following 
\begin{theorem}[Weak arithmetic Artin-Takagi decomposition, \cite{PJwon}] \label{AATD}
\[n(G, \emph{Reg}_G) = \sum_{\chi \in \emph{Irr}(G)} \chi(1)n(G, \chi).\]
\end{theorem} 
\begin{proof} Just follows from definition \ref{WAC} and the fact that $\mbox{Reg}_G = \sum_{\chi \in \mbox{Irr}(G)} \chi(1)\chi$
\end{proof}
\noindent Like always, we should again have compatibility between $\Theta_{G}$ and $\Theta_{H}$'s. 
\begin{theorem} [Weak arithmetic Heilbronn-Stark Lemma, \cite{PJwon}] \label{WAHSL}
For every cyclic subgroup $H$ of $G$, we have
\[\Theta_G|_{H} = \Theta|_{H}.\] 
\end{theorem}
\begin{proof} Proof is just analogous to the arguments we gave for classical Heilbronn case at Theorem \ref{HSL}. We only need to use Frobenius reciprocity and first two conditions of Definition \ref{WAC}.
\end{proof}
\noindent Like the classical Heilbronn-Stark lemma (Theorem \ref{HSL}), above theorem enables us to bound $n(G, \chi)'s$ by the following theorem. This inequality is variant of Foote-Murty inequality \ref{MRI}. 
\begin{theorem} [Wong, \cite{PJwon}] \label{AFMI}
\[\sum_{\chi \in \emph{Irr}(G)} n(G, \chi)^2 \leq n(G, \emph{Reg}_G)^2.\]
\end{theorem} 
\begin{proof} By the orthogonality property of characters and the definition of $\Theta_G$ one can say,
\[\sum_{\chi \in \mbox{Irr}(G)}n(G, \chi)^2= (\Theta_G, \Theta_G) = \frac{1}{|G|}\sum_{g\in G}|\Theta_G(g)|^2\]
Arithmetic version of Heilbronn-Stark lemma (Theorem \ref{WAHSL}) gives
\[\Theta_G(g)=\Theta_{\langle g \rangle}(g) = \sum_{\phi \in \mbox{Irr}(g)} n(\langle g \rangle, \phi)\phi(g).\]
As $\langle g \rangle$ is cyclic, all of $n(\langle g \rangle,\phi)$s are non negative, and so
\[|\Theta_G(g)| \leq n(\langle g \rangle, \phi) =n\Big(\langle g \rangle, \sum_{\phi \in \mbox{Irr}(\langle g \rangle)} \phi\Big).\]
But later one is equals to 
\[n(\langle g \rangle, \mbox{Reg}_{\langle g \rangle}) = n(G, \mbox{Reg}_G),\] 
because $\mbox{Ind}_{\langle g \rangle}^{G} \mbox{Reg}_{\langle g \rangle}=\mbox{Reg}_G$.
\end{proof} 
\noindent As an immediate corollary one can easily obtain the following analogues version of theorems discussed in the previous chapter. 
\begin{corollary}[Weak arithmetic Aramata-Brauer Theorem, \cite{PJwon}] \label{BAT}
\[|n(G, \emph{Reg}_G)| \geq  |n(G, 1_G)|,\]
where $1_G$ denotes the trivial character of $G$.
\end{corollary}
\begin{proof} Just take $\chi=1_G$ in Theorem \ref{AFMI}.
\end{proof}
\begin{corollary}[Weak arithmetic Stark Lemma, \cite{PJwon}] If $n(G, \emph{Reg}_G) \leq 1$, then $n(G, \chi) \geq 0$ for all irreducible characters $\chi$ of $G$.
\end{corollary} 
\begin{proof} If $n(G, \mbox{Reg}_G) \leq 0$, then Theorem \ref{AFMI} directly says that  $n(G,\chi)=0$ for any $\chi \in \mbox{Irr}(G)$. But if $n(G, \mbox{Reg}_G) = 1$, from above theorem we can say $|n(G,\chi)|=1$ for at most one $\chi$ and all others are zero. Take a $\chi_0 \in \mbox{Irr}(G)$ such that $|n(G,\chi_0)|=1$. By weak arithmetic Artin-Takagi decomposition \ref{AATD} one can write,
\[\sum n(G,\chi)\chi(1)=n(G,\mbox{Reg}_G)=1,\]
and so we must have $\chi_0(1)=1$ as well as $n(G,\chi_0)=1$, as desired.
\end{proof} 
\noindent After this our goal is to further generalize this characters and use it to study several $L$-functions. But before that we would like to introduce the following formalism.\\ 
\newline
Let $\rho$ be some character of a subgroups $H$ of $G$. We call $\rho \otimes \phi$ to be twist of $\rho$ by $\phi$, where $\phi$ is some character of $H$ and $(\rho \otimes \phi)(h)=\rho(h)\phi(h)$ for all $h\in H$. We will explain importance of twists later. In this context we first have the following result. 
\begin{theorem} [Wong, \cite{PJwon}] \label{TWIST} Let $\rho$ be an arbitrary character of $G$. Suppose that for every cyclic subgroup $H$ of $G$ and irreducible character $\phi$ of $H$, we have $n(H, \rho|_{H} \otimes \phi) \geq 0$, then 
\[ \sum_{\chi \in \emph{Irr}(G)} n(G, \rho \otimes \chi)^2 \leq n(G, \rho \otimes \emph{Reg}_G)^2.\]
\end{theorem}  
\begin{proof} Proof is very predictable from nature of the statement. Indeed, for every cyclic subgroup $H$ of $G$ (or $H= G$) and every character $\phi$ of $H$, we define $n'(H,\phi)$ to be $n(H,\rho|_{H}\otimes \phi)$. Note that tensor product respects linearity, and by the hypothesis made in the statement of theorem we can see $n'(H,\phi)$'s satisfy first and last conditions of Definition \ref{WAHSL}. So to prove they are actually weak arithmetic Heilbronn characters, we only need to justify second condition. Note that
\[n'(H, \phi) = n(H, \rho|_{H} \otimes \phi) = n(G,\mbox{Ind}^G_H(\rho|_{H} \otimes \phi)),\]
and from the fact that tensoring commutes with induction, one can see the later one is equals to 
\[n(G, \rho \otimes \mbox{Ind}_{H}^{G} \phi)=n'(G, \mbox{Ind}_{H}^{G} \phi),\]
as desired. The proof now follows from Theorem \ref{AFMI}.
\end{proof} 
\subsection{Introduction to Arithmetic Heilbronn characters and variant of Uchida-Van der Waal} \label{ACH}
\noindent In this section we will put more conditions on $n(H, \phi)$'s, which will enable us study several other L-functions, comparing with the results from previous chapter.
\begin{definition} \label{AHC} Suppose that there is a set of integers $\{n(H, \phi)\}_{(H,\phi)\in I(G)}$, where \[I(G)=\{(H, \phi) \mid H \text{ is a subgroup of } G, \text{ and } \phi \text{ is a character of $H$}\}\]
satisfying the following three properties:
\begin{itemize}
\item \textbf{ACH1:} $n(H, \phi_1 + \phi_2) = n(H, \phi_1) + n(H, \phi_2)$ for any characters $\phi_1$ and $\phi_2$ of $H$,
where $H$ is a subgroup or an improper subgroup of $G$.
\item \textbf{ACH2:} $n(G,\emph{Ind}^G_H \phi) = n(H, \phi)$ for every subgroup $H$ and every character $\phi$ of $H$. 
\item \textbf{ACH3:} $n(H, \phi) \geq 0$ for all subgroups $H$ of $G$ and all one dimensional characters $\phi$ of $H$.
\end{itemize} 
\noindent Then arithmetic Heilbronn character of a subgroup $H$ of $G$, is defined by
\[\Theta_H = \sum_{\phi \in \mbox{Irr}(H)} n(H, \phi)\phi.\]
\end{definition}
 \begin{rem} Note that arithmetic Heilbronn characters are weak arithmetic Heilbronn characters as well. So any result holding with weak arithmetic Heilbronn characters will hold with arithmetic Heilbronn characters too. 
\end{rem}
\noindent Similarly as before, we have the following result,
\begin{theorem}[Arithmetic Heilbronn-Stark Lemma, \cite{PJwon}] \label{AHSL} For every subgroup $H$ of G, one has
\[\Theta_G|_{H} = \Theta|_{H}.\]
\end{theorem}  
\noindent From now on, we will always denote $\Theta_{H}$ as the arithmetic Heilbronn character associated to $H$, unless otherwise specified. Suppose $G$ is solvable and $H$ be its subgroup. Then we can expect some variant of the result by Uchida and van der Waal \ref{UVdW}. Indeed we can get a analogous result, which we will establish shortly. But before that let us establish one useful result. 
\begin{proposition} \label{acimp} Let $G$ be a solvable group, and $H$ be a subgroup of $G$. Let $\chi$ and $\phi$ be $1$-dimensional characters 
of $G$ and $H$, respectively. Then 
\[n(G,\emph{Ind}^G_H \phi) \geq (\chi|_{H}, \phi)n(G, \chi).\]  \end{proposition} 
\begin{proof} First of all if $(\chi|_H, \phi) = 0$, then we are done by $\textbf{ACH3}$. So we may take $(\chi|_H, \phi) > 0$. As both of $\chi$ and $\phi$ are one 1-dimensional, it is evident that $\chi|_H = \phi$ and $(\chi|_H, \phi) = 1$. From Corollary \ref{imp3}, we can write 
\[\mbox{Ind}^G_H(1|_{H}) = 1_G + \sum_{i=1}^{n} \mbox{Ind}^G_{H_i} \phi_i.\]
Twisting both sides with $\chi$ one obtains
\[\mbox{Ind}^G_H \phi=\mbox{Ind}^G_H (\chi|_H) = \chi + \sum_{i} \mbox{Ind}^G_{H_i} (\chi|_{H_i} \otimes \phi_i),\] and the proof follows trivially from $\textbf{ACH3}$, as $n(H_i, \chi|_{H_i} \otimes \phi_i) \geq 0$, for any $1 \leq i\leq n$.
\end{proof}
\noindent As a corollary we can easily deduce the following result, just taking $\chi$ and $\phi$ to be trivial characters.
\begin{corollary}[Arithmetic Uchida-Van der Waall, \cite{PJwon}] \label{AUVdW} Let $G$ be a solvable group,
and $H$ be a subgroup. Then
\[n(G,\emph{Ind}^G_H 1_H)-n(G, 1_G) \geq 0.\]
\end{corollary}
\begin{rem} One can see this is indeed a variant of Uchida-Van der Waall's theorem taking $G=\mbox{Gal}(\overline{K}/F), H=\mbox{Gal}(\overline{K}/K)$ and $n(H',\phi)=\mbox{ord}_{s=s_0}L\Big(s,\phi,\overline{K}/K^{H'}\Big)$, for any subgroup $H'$ of $G$.
\end{rem} 
\subsection{Variants of Murty-Raghuram's results}
From now on we will assume $G$ is solvable and $H$ be its subgroup. In this new settings we have the following analogous result to Murty-Raghuram's inequality for Heilbronn characters.
\begin{theorem} [Wong, \cite{PJwon}] \label{M-R I2} Let $\chi_0$ be a $1$-dimensional character of $G$. Then
\[ \sum_{\chi \in \emph{Irr}(G)-\{\chi_0\}} n(G, \chi)^2 \leq \Big(n(G, \emph{Reg}_G)-n(G, \chi_0)\Big)^2.\]
\end{theorem}
\begin{proof} We define a truncated arithmetic Heilbronn character with respect to $\chi_0$ as follows,
\[\Theta^{\chi_0}_{G}=\sum_{\chi \in  \mbox{Irr}(G)-\{\chi_0\}}n(G, \chi)\chi,\]
\noindent which leads us to conclude 
\[\sum|\Theta_G^{\chi_0}|^2 = \frac{1}{|G|} \sum_{g\in G}|\Theta^{\chi_0}_G (g)|^2 = \sum_{\chi \in \mbox{Irr}(G)-\{\chi_0\}} n(G, \chi)^2.\]
From arithmetic Heilbronn–Stark lemma \ref{AHSL} we can write 
\[\Theta^{\chi_0}_G (g)=\Theta_G(g)-n(G, \chi_0)\chi_0(g)=\Theta_{\langle g \rangle}(g)-n(G, \chi_0)\chi_0(g),\]
and later one is equals to
\[\sum_{\phi \in \mbox{Irr}(\langle g \rangle)}\Big(n(\langle g \rangle, \phi)-n(G, \chi_0)(\chi_0|_{\langle g \rangle}, \phi)\Big)\phi(g).\]
\noindent In Proposition \ref{acimp} taking $H=\langle g \rangle$ and $\phi \in \mbox{Irr}(\langle g \rangle)$ we can simply write 
\begin{align}
|\Theta^{\chi_0}_G (g)| & \leq \sum_{\phi \in \mbox{Irr}(\langle g \rangle)}\Big(n(\langle g \rangle, \phi)-n(G, \chi_0)(\chi_0|_{\langle g \rangle}, \phi)\Big) \\
& = n(\langle g \rangle, \mbox{Reg}_{\langle g \rangle})-n\big(G,(\chi_0|_{\langle g \rangle}, \mbox{Reg}_{\langle g \rangle})\chi_0\big), 
\end{align}
but since $\chi_0$ was taken to be one dimensional, $\chi_0|_{\langle g \rangle}=1$, and so $\chi_0|_{\langle g \rangle} \in \mbox{Irr}(\langle g \rangle)$. Thus, $(\chi_0|_{\langle g \rangle}, \mbox{Reg}_{\langle g \rangle})=1$. Also we know $n(\langle g \rangle, \mbox{Reg}_{\langle g \rangle})=n(G, \mbox{Reg}_{G})$, from $\textbf{ACH2}$. The proof is therefore completed.
\end{proof} 
We are not proving next two theorems. They are exactly analogous to Theorem \ref{RR2} and Lemma \ref{MRI 2}. For a more detailed proof, see Theorem 3.8 and 3.9 in \cite {PJwon}.
\begin{theorem} Let $H$ be a subgroup of $G$, and let $\phi$ be any 1-dimensional character of $H$. Let $S^{\phi}$ denote the set of all $1$-dimensional characters of $G$ whose restrictions on $H$ are $\phi$. Then
\[n(G,\emph{Ind}^G_H \phi)-\sum_{\chi \in S_{\phi}} n(G, \chi) \geq 0.\]
\end{theorem} 
\begin{theorem}Let $S$ be the set of all $1$-dimensional characters of $G$. Then
\[\sum_{\chi \in \emph{Irr}(G)-S}n(G, \chi)^2 \leq \Big(n(G, \emph{Reg}_G)-n\big(G,\emph{Ind}^G_{G^1} 1_{G^1}\big)\Big)^2.\]
\end{theorem}
\subsection{Variants of Lansky-Wilson's work}
\noindent In this section we adapt the method developed by Lansky-Wilson, and show how to generalize of Murty and Raghuram's work in the setting of arithmetic Heilbronn characters. Our first result is,
\begin{theorem}[Wong, \cite{PJwon}] Let $d$ be the greatest common divisor of the degrees of the characters in $\emph{Irr}(G)-S^i$, where $S^i$ denotes the set of irreducible characters of $G$ of level less than or equal to $i$. Then 
\[n(G, \emph{Reg}_G)-n(G,\emph{Ind}_{G_i}^{G}1_{G^i})=kd,\] for some non-negative integer $k$.
\end{theorem}
\begin{proof} From conditions  $\textbf{AHC1},\textbf{AHC2}$ and Lemma \ref{imp for i}, we get
\[n(G, \mbox{Reg}_G)-n(G,\mbox{Ind}^G_{G_i}1_{G_i}) = n(G, \mbox{Reg}_G)-\sum_{\chi \in S^i} \chi(1)n(G, \chi),\]
but the later one is $\sum_{\chi \in \mbox{Irr}(G)-S^i}\chi(1)n(G, \chi)$, which is clearly a multiple of the greatest common divisor of the degrees of the characters $\chi$ of G with $l(\chi) > d$. Using the fact that $\mbox{Ind}_{G^i}^{G}(\mbox{Reg}_{G^i})=\mbox{Reg}_G$, we have
\[n(G, \mbox{Reg}_G)-n(G,\mbox{Ind}^G_{G_i}1_{G_i})= n(G^i,\mbox{Reg}_{G^i})-n(G^i, 1_{G^i}),\]
and from weak arithmetic Aramata-Brauer theorem \ref{BAT}, it  follows that 
\[n(G^i,\mbox{Reg}_{G^i})-n(G^i, 1_{G^i}) \geq 0.\]
So the proof is complete. 
\end{proof}
The next result generalizes Lansky-Wilson's main work in \cite{LW}.
\begin{theorem}[Wong, \cite{PJwon}] \label{LW5} Let $\psi$ be a 1-dimensional character of a subgroup $H$ of G. Then
\[n(G,\emph{Ind}^G_H \psi)-\sum_{\chi \in S^i}(\chi,\emph{Ind}^G_H\psi)n(G, \chi) \geq 0.\]
\end{theorem}
\begin{proof} Proof is exactly same as proof of Theorem \ref{LWT1}. 
\end{proof}
\begin{corollary} Let $\phi_0$ be a $1$-dimensional character of a subgroup $H$ of $G$, and $S^i_{\phi_0}$ be the set of irreducible characters of level $i$ occurring in $\emph{Ind}^G_H \phi_0$. Then
\[\sum_{\chi \in S^i_{\phi_0}}(\chi,\emph{Ind}^G_H \phi_0)n(G, \chi) \geq 0.\]
\end{corollary}
\begin{proof} If $\phi_0$ is non-trivial on $H \cap G^i$, then by similar argument as second part of proof of Theorem \ref{LWT1} gives $(\chi,\mbox{Ind}^G_H \phi_0)=0$ for all $\chi \in S^i_{\phi_0}$, and so the result follows immediately.
Otherwise, $\phi_0$ extends uniquely to a character $\phi$ of $HG^i$. Then from Theorem \ref{LW5} we see that
\[n(G,\mbox{Ind}^G_{HG^i} \phi)-\sum_{\chi \in S^{i-1}}(\chi,\mbox{Ind}_{HG^i}^{G}\phi)n(G,\chi)\geq 0,\]
and just by Lemma \ref{imp for i}, right hand side is always non negative due to \textbf{ACH3} and left hand side above is equals to
\[\sum_{\chi \in S^i_{\phi_0}}(\chi,\mbox{Ind}^G_H \phi_0)n(G, \chi).\] 
Therefore the result follows.
\end{proof}
\begin{theorem}[Wong, \cite{PJwon}] Let $K/F$ be a solvable extension of number fields, and let $G=\emph{Gal}(K/F)$, then 
\[\sum_{\chi \in \emph{Irr}(G),l(\chi)=i} n(G,\chi)^2 \leq \Big(n(G,\emph{Ind}^G_{G^i} 1_{G^i})-n(G,\emph{Ind}^G_{G^{i-1}}1_{G^{i-1}})\Big)^2.\]
\end{theorem}
\begin{proof} Again we follow proof of Theorem \ref{LWi}. Based on the notations in the proof, we can write
\[|\Theta^i_{G}(g)|\leq n\Big(G,\sum_{l(\chi)=i} \chi(1)\chi\Big).\]
Now, 
\[n\Big(G,\sum_{\chi \not\in S^i}\chi(1)\chi\Big)=n(G,\mbox{Reg}_G)-n\Big(G,\mbox{Reg}_{G_i}^{G}1_{G^i}\Big),\]
also,
\[n\Big(G,\sum_{l(\chi)=i} \chi(1)\chi\Big)=n\Big(G,\sum_{\chi \not\in S^i}\chi(1)\chi\Big)-n\Big(G,\sum_{\chi \not\in S^{i-1}}\chi(1)\chi\Big).\]
and so does the proof follows evidently. 
\end{proof}
The next result is variant to Corollary \ref{crr}. 
\begin{theorem}[Wong, \cite{PJwon}] \label{nice}Let $G$ be a solvable group. Then $n(G, \emph{Reg}_G)-n(G,\emph{Ind}^G_{G^i} 1_{G^i})$ 
cannot be $1$.
\end{theorem}
\begin{proof} See \cite{PJwon} pp. 1559, Corollary 3.16.
\end{proof}
\begin{rem}
Let us get back to Remark \ref{rem5}, we again remark that the possible improvement we hoped for has been recently shown by Wong in \cite{PJwon}, even in this new set up. See Theorem 3.15 in \cite{PJwon} for a proof.
\end{rem}
\noindent Finally time has come to use our arithmetic Heilbronn characters to study Artin L-functions, as we aimed for at the beginning. Which
we are going to do from next section onward. 
\subsection{Application on L-functions}
\noindent Before going further let us recall what exactly we just did before. We started with work of Foote-Murty, then we discussed work of Murty-Raghuram. We saw Lansky-Wilson generalized their results in 2005. Recently Wong has generalized many of their results in a different settings. Wong did it to apply similar kind of results on several other L-functions, and not just Artin L-functions. Key idea was of course, notion of artihmetic Heilbronn characters. We will discuss how to use this for other L-functions. To enter the world of L-functions, the first result that we are going to prove is the following. 
\begin{theorem}[Wong, \cite{PJwon}] \label{AHC to L} Let $K/F$ be a solvable Galois extension of number fields and $G=\emph{Gal}(K/F)$, and let $\rho$ be a $2$-dimensional representation of $G$. Then for any subgroup $H$ of $G$, the quotient 
\[\frac{L(s,\emph{Ind}^G_H \rho|_H, K/F)}
{L(s, \rho, K/F)}\] 
is holomorphic on $\mathbb{C}-\{1\}$. Moreover, for every $1-$dimensional character $\chi_0$ of $G$, one has
\[\sum_{\chi \in \emph{Irr}(G)-\{\chi_0\}}(\emph{ord}_{s=s_0} L(s, \rho \otimes \chi))^2 \leq \Big(\emph{ord}_{s=s_0} \Big(\frac{\zeta^2_K(s)}{L(s,\rho \otimes \chi_0,K/F)}\Big)\Big)^2.\]
\end{theorem}  
\noindent Before going into proof let us look at statement of the theorem very carefully, $\rho$ is assumed to be $2$-dimensional. Why ? well, the deep result by Langlands and Tunnell \cite{LT} says, for all 2-dimensional irreducible representations $\rho$ of $G, L(s,\rho,K/F)$ is entire. Yes, this is the main reason behind the statement. As a consequence of  Langlands-Tunnell theorem, we can say for any  $2$-dimensional representation $\rho$ of $G$ and any one dimensional character $\phi$ of a subgroup $H$ of G, the Artin L-function $L(s, \rho|_H \otimes \phi, K/K^H)$ is entire, because $\rho|_H \otimes \phi$ is also 2-dimensional. Which says $n'(H,\phi)$'s are all non negative, where 
\[n'(H,\phi)=\mbox{ord}_{s=s_0} L(s, \rho|_H \otimes \phi, K/K^H).\]
From proof of Theorem \ref{TWIST} we see they are arithmetic Heilbronn characters, and this is what we need to prove Theorem \ref{AHC to L}. Here is the proof,
\begin{proof} Our proof goes by light of the formalism established in Theorem \ref{TWIST}. First we note that
\[\rho \otimes \mbox{Reg}_G=\rho \otimes \mbox{Ind}^{G}_{\langle e_G \rangle} 1_G = \mbox{Ind}^{G}_{\langle e_G \rangle} \big(\rho \otimes 1_G\big),\]
as tensoring commutes with induction. But since $\rho$ is 2-dimensional, later one is $2\mbox{Ind}^{G}_{\langle e_G \rangle} 1_G=2\mbox{Reg}_G$. Now just use Proposition \ref{acimp} and Theorem \ref{M-R I2}.
\end{proof}
\noindent We will now see application of this tool to study other L-functions, and weapon will always be to choose suitable integers to construct arithmetic Heilbronn characters. We start with introducing the concept of Artin-Hecke L-functions developed by Weil in 1979. First recall what Hecke characters are, we defined \textit{classical} Hecke characters in Remark \ref{Hecke}. In iedelic settings it is defined to be a continous map from ring of ideles $\mathbb{A}_F^{*}$ to $\mathbb{C}$, whose kernal contains $F^{*}$. See Chapter IV of \cite{NK2} for a more discussion on ideles. It can be shown from topology on $\mathbb{A}_F^{*}$ and continuity of Hecke characters that, for any Hecke character $\psi$, and for all but finitely many prime ideal $v$ of $O_F$, $\psi(O^{*}_v)=1$. This allows us to define a homomorphism on $I^{\mathfrak{m}}_F$ (for some modulus $\mathfrak{m})$) by $\psi(v)=\psi|_{v}(\pi_v)$, where $\psi|_{v}$ is the restriction of $\psi$ to $F^{*}_v$ and $\pi_v$ is a uniformizer of $F_v$. From now on, by Hecke character we will always mean that it is defined in idelic settings.
\begin{definition} \label{AHLF}Let $K/F$ be a Galois extension of number fields with Galois group
$G$. Let $\psi$ be a Hecke character of $F$ and $\rho$ be a complex representation of $G$ with underlying vector space $V$. The Artin–Hecke L-function attached to $\psi$ and $\rho$ is defined by
\[L(s, \psi \otimes \rho, K/F) = \prod_{\mathfrak{p} \subseteq O_F} \det\Big(1-\psi(\mathfrak{p})\rho(\sigma_{\mathfrak{P}})N(\mathfrak{p})^{-s}|_{{V^{I_{\mathfrak{P}}}}}\Big)^{-1},\]
\noindent where the product runs over prime ideals in $\mathfrak{p}$ in $O_F$, and $\mathfrak{P}$ denotes a prime ideal in $O_K$ lying above $\mathfrak{p}, I_{\mathfrak{P}}$ is the inertia subgroup at $\mathfrak{P}$. And $V^{I_{\mathfrak{P}}}$ is set of vectors of $V$ fixed by $I_{\mathfrak{P}}$. For our convenience we we denote $L(s, \psi \otimes \chi, K/F)$ as $L(s, \psi \otimes \rho, K/F)$ where $\chi$ the character for $\rho$. 
\end{definition}
\noindent We recall if $\chi$ is one dimensional then $L(s,\chi,K/F)$ can be associated to a Hecke L-function, and correspondence is given by $H(\chi)(\mathfrak{p})=\chi(\sigma_{\mathfrak{p}})$. So, for any Hecke character $\psi, L(s,\psi \otimes \chi, K/F)$ can be identified with the Hecke L-function $L(s,\psi \otimes H(\chi),K/F)$, and the later one is known to have analytic continuation except possibly a pole at $s=1$. Now one may naturally ask the following question. 
\begin{question} For any irreducible character $\chi$ of $\mbox{Gal}(K/F)$ and Hecke character $\psi$ of $F$, does the twisted L-function $L(s,\psi \otimes \chi, K/F)$ has analytic continuation to $\mathbb{C}-\{1\}$, with a possible pole at $s=1$ ?
\end{question}
\noindent Of course we do not know yet, because this is stronger version of Artin's holomorphy conjecture at \ref{ACc}. But we can indeed get their meromorphic continuation, from the following lemma and Brauer's induction theorem or Theorem \ref{mero}.
\begin{lemma} \label{Wei} Let $H$ be a subgroup of $G$. Then for any characters $\chi_1, \chi_2$ of $G$ and $\phi$ of $H$ and Hecke character $\psi$ of $G$, we have 
\begin{itemize}
\item $L(s, \psi \otimes (\chi_1 + \chi_2), K/F) = L(s, \psi \otimes \chi_1, K/F)L(s, \psi \otimes \chi_2, K/F)$
\item $L(s, \psi \otimes \emph{Ind}^G_H \phi, K/F) = L(s, (\psi \circ N_{K^H/F}) \otimes \phi, K/K^H)$
\end{itemize}
\noindent where $K^H $is a subfield of $K$ fixed by $H$ and $N_{K^H/F}$ is the usual norm of $K^H/F$.\\
\end{lemma}
\begin{proof} This is not difficult to prove after knowing Lemma \ref{first}. For more details see \cite{Weil}.
\end{proof} 
\noindent We consider a non trivial Hecke character of infinite order $\psi$, associated to $F$. Fix a point $s_0 \in \mathbb{C}$ and set
\[n^{\psi}(H,\phi)=\mbox{ord}_{s=s_0} L(s, (\psi \circ N_{K^H/F}) \otimes \phi, K/K^H).\]
\noindent One can see these $n^{\psi}(H, \phi)$'s are non negative when $H$ is a subgroup of $G$ and $\phi$ is a 1-dimensional character on $H$, at $s=s_0\neq 1$. This is because $(\psi \circ N_{K^{H}/F}) \otimes \phi$ is a Hecke character of $K^H$. At $s_0=1$, L-function associated to a trivial Hecke L-function has a simple pole, so $n(H,\phi)$ is negative if and only if $(\psi \circ N_{K^H/F}) \otimes \phi$ is trivial. Which means $\psi$ is of finite order. So for any $s=s_0$ these set of integers $n^{\psi}(H, \phi)$ defined above gives arithmetic Heilbronn characters, a special thanks to Weil's Lemma \ref{Wei}. Doing similar things as before we can deduce analogous results to all the theorems we discussed in section \ref{ACH}.\\
\newline 
\noindent One may think what is the reason behind twisting characters with Hecke characters. Astonishingly it is used to study $L$-functions for elliptic curves. And this idea was due to M. Murty and K. Murty \cite{MME}, based on Deuring's work \cite{Deu}. We first introduce to L-functions for elliptic curves.
\begin{definition} Let $K$ be a number field and $E$ be an elliptic curve over $K$. We define the $L$-series attached to $E$ by
\[L_{K}(s)=\prod L_{v}(s),\]
where $v$ runs through all finite places of $K$ coprime to conductor (see \cite{Sil} pp. 450 for definition) of $E/K$. Local factors $L_v(s)$'s are defined by
\[L_v(s)=\big(1-a_v(N(v))^{-s}+(N(v))^{1-2s}\big)^{-1},\]
where $a_v$'s are given by $N(v)+1-|E \mod v|$, and $N(v)$ is being index of $v$ in $O_K$. Here the product is over finite primes $v$ of $K$.
\end{definition}
We consider the endomorphism ring $\mbox{End}_{\overline{K}}(E_{/K})$. There are always two possibilities, either this is $\mathbb{Z}$, in which case $E$ is said be without complex multiplication (CM) over $K$. Or, the endomorphism ring an order in an imaginary quadratic field, say $F$. In which case $E$ is said to have complex multiplication, or CM in short, over $K$. For more details about CM elliptic curves and definition of order, see pp. 100-102 in \cite{Sil}.  
\begin{conjecture}[Birch and Swinnerton-Dyer, \cite{Sil}]\label{BSD} $L_K(s)$ extends to en entire function, satisfying a suitable functional equation, and at $s=1$ has a zero of order equals to rank of $E(K)$. Where $E(K)$ is group of $K$-rational points on $E$. 
\end{conjecture}
\noindent Here we will mainly discuss about quotient of such $L$-functions. In this context, M. Murty and K. Murty proved the following in \cite{MME}.  
\begin{theorem} \label{MME1} Let $E$ be an elliptic curve defined over $K$. Suppose that $E$ has \emph{CM} by an order in an imaginary quadratic field $F$. Let $M$ be a finite Galois extension of $K$, then $\frac{L_M(s)}{L_K(s)}$ is entire. Moreover if $M$ is contained in a solvable extension of $K$, then $\frac{L_M(s)}{L_K(s)}$ is entire. 
\end{theorem} 
\noindent Game of twisting a character by a Hecke character, comes to prove above result. We will give a short sketch of proof to understand why the concept of twisting has been crucial. From Deuring's result \cite{Deu} below,
\begin{theorem}[Deuring, \cite{Deu}]Let $E$ be an elliptic curve defined over $K$. Suppose that $E$ has \emph{CM} by an order in an imaginary quadratic field $F$. If $F \subseteq K$, then the $L$-function $L_K(s)$ of $E$ can be written as $L(s,\psi_{K})L(s,\overline{\psi_{K}})$, where $\psi_K$ is certain Hecke character of $K$. And if $F\not\subseteq K$, then $L_K(s)$ is equal to $L(s,\psi_{FK})$, where $\psi_{FK}$ is certain Hecke character of $FK$.
\end{theorem} 
\begin{rem} Here $\psi_M$'s are certain Hecke characters of number field $M$'s, defined at page 168 of \cite{Sil2}. Deuring defined $\psi_K$ in such a way, so that local factors of $L_K(s)$ and of $L(s,\psi_{K})L(s,\overline{\psi_{K}})$ matches. Which means $\psi_K(v)\overline{\psi_K(v)}=N(v)$, for all but finitely many prime ideals $v \subseteq O_K$. Which means, $\psi_K$'s can not have finite order. We will use this fact lot of times throughout rest of this section. 
\end{rem}
\noindent M. Murty and K. Murty derived the following result.
\begin{theorem}[Murty-Murty, \cite{MME}] Every L-function of a \emph{CM}-elliptic curve can be written in terms of Hecke L-functions. 
\end{theorem} 
Now let us sketch proof of Theorem \ref{MME1} in details. \\
\newline  
\textit{Proof of Theorem \ref{MME1}.} First we consider the case, $F \subseteq K$, where $F$ is the imaginary quadratic extension of $\mathbb{Q}$ mentioned as mentioned earlier. In this case we define 
\[L_K(s,\phi)=L\Big(s,(\psi_{K}\circ N_{M^{H}/K})\otimes \phi, M/M^{H}\Big)L\Big(s,(\overline{\psi_{K}}\circ N_{M^{H}/K})\otimes \phi, M/M^{H}\Big),\]
for any subgroup $H$ of $\mbox{Gal}(M/K)$ and character $\phi$ on $H$. We also define $n(H,\phi)$ to be the order of zero of above, at $s=s_0$. From Lemma \ref{Wei} and the fact that $\psi_K$ does not have finite order, it follows that $n(H,\phi)'s$ satisfy all the required conditions of Definition \ref{AHC}. Now if $M$ is any arbitrary extension of $K$, we can look $E$ over $M$. When we do so, we will denote this as $E/_{M}$. First part of the theorem follows from Corollary \ref{BAT}, because $n(H,\phi)$'s are always non negative, for any character $\phi$ on $H$, and $L_K(s,\mbox{Reg}_G)$ corresponds to $L_M(s)$. Former one is true because L-functions of CM elliptic are entire, as they can be written as product of one or two Hecke L-functions of infinite order. Later one is true because, by work of Serre \cite{Serrel} the $L$-function of $E$ over $M$ is given by the family of $\ell$-adic representations 
\[\rho_{\ell}:\mbox{Gal}(\overline{K}/M) \longrightarrow  \mbox{Aut}\big(T_{\ell}(E/_{M})\big),\] 
where $\overline{K}$ is algebraic closure of $K$ and $T_{\ell}(E/_{M})$ is denoted by inverse limit of $E/_{M}[\ell^n]$'s, where $E/_{M}[\ell^n]$ is set of $\ell^n$ torsion points of $E$ lying in $M$. Since $T_{\ell}(E/_{M}) = T_{\ell}(E/_{K})=\mbox{GL}_2(\mathbb{Z}_{\ell})$ as $\mbox{Gal}(\overline{K}/M)$ modules, it follows that $\rho_{\ell}$ is the restriction of the representation
\[\psi_{\ell} : \mbox{Gal}(\overline{K}/K) \longrightarrow \mbox{Aut}\big(T_{\ell}(E/_{K})\big).\]
So
\[L(s,\rho_{\ell})=L\Big(s,\psi_{\ell}|_{\mbox{Gal}(\overline{K}/M)}\Big),\]  
but 
\[\mbox{Ind}_{H'}^{G'}(\psi_{\ell}|_{H'})=\psi_{\ell} \otimes \mbox{Ind}_{H'}^{G'} 1_{H'}=\psi_{\ell} \otimes \mbox{Reg}_{G'/H'},\]
where $G'=\mbox{Gal}(\overline{K}/K)$ and $H'=\mbox{Gal}(\overline{K}/M)$. And so $\mbox{Reg}_{G'/H'}=\mbox{Reg}_{G}$, as $G'/H'=G$. And this Completes first part of the proof using arithmetic Brauer-Aramata Theorem (Theorem \ref{BAT}), for the case $F\subseteq K$.\\ 
\newline 
For the other case, following similar procedure we can show $\frac{L_{MF}(s)}{L_{KF}(s)}$ is entire. Now, 
\[\frac{L_{KF}(s)}{L_K(s)}=\frac{L(s,\psi_{FK})L(s,\overline{\psi_{FK}})}{L(s,\psi_{FK})}.\]
But the later one is entire, because $\psi_{FK}$ has infinite order, so $\frac{L_{MF}(s)}{L_{K}(s)}$ is entire. If $F\subseteq M$, then $MF=M$, and we are done. Now if $F\not\subseteq M$, then $F\cap M=\mathbb{Q}$. In this case, $\mbox{Gal}(M/K)$ is isomorphic to $\mbox{Gal}(MF/KF)$, we denote this isomorphism to be $f$. Now for any subgroup $H$ of $\mbox{Gal}(M/K)$ and one dimensional character $\phi$ on $H$, define  \[n(H,\phi)=\mbox{ord}_{s=s_0}L\Big(s, (\psi_{FK} \circ N_{(MK)^{f(H)}/(FK)})\otimes \chi \circ f^{-1}, MK/(MK)^{f(H)}\Big),\]
and follow the same argument as we provided for the case $F \subseteq K$.\\
\newline
For second part of the theorem, let $\overline{M}$ be Galois closure of $M$ over $K$, and $G$ be the Galois group. In this we only need to realize that $L_F\Big(s,\mbox{Ind}_{H}^{G}(1|_H)\Big)=L_{\overline{M}^{H}}(s)$, for any subgroup $H$ of $G$. Which can be done by the similar argument, as we gave to show $L_F(s,\mbox{Reg}_G)=L_M(s)$ in the previous part. After this the proof is immediate using Corollary \ref{AUVdW}.\\
\newline 
Here we use the power of arithmetic Heilbronn characters to get elliptic analogue of the Uchida–van der Waall theorem as follows, 
\begin{theorem}[Wong, \cite{PJwon}] \label{5.27}Suppose $K/F$ is a solvable Galois extension with Galois group $G$, and let $H$ be a subgroup of $G$. Let $\chi$ and $\phi$ be $1$-dimensional characters of $G$ and $H$, respectively. Then
\[\frac{L_F(s,\emph{Ind}^G_H \phi)}{L_F(s, \chi)^{(\chi|_H,\phi)}}\] 
is entire. 
\end{theorem} 
Moreover the elliptic analogue of Murty–Raghuram's inequality \ref{M-R I}, is as follows. 
\begin{theorem}[Wong, \cite{PJwon}] \label{5.28} For every $1$-dimensional character $\chi_0$ of G, 
\[\sum_{\chi \in \emph{Irr}(G)-\{\chi_0\}}\Big(\emph{ord}_{s=s_0} L_F(s, \chi)\Big)^2 \leq \Big(\emph{ord}_{s=s_0} \frac{L_K(s)}{L_F(s, \chi_0)}\Big)^2.\]
\end{theorem} 
\begin{proof} Proof of Theorem \ref{5.27} is immediate from Proposition \ref{acimp}, and Theorem \ref{5.28} is immediate from Theorem \ref{acimp}, where $n(H,\phi)$'s are defined before. 
\end{proof} 
We have below an interesting result for L-functions of CM-elliptic curves by just applying Theorem \ref{nice} and Theorem \ref{MME1}.
\begin{theorem}[Wong, \cite{PJwon}] Suppose $K/F$ is a solvable Galois extension with Galois group $G$. Then for all $i \geq 1$,
\[\frac{L_K(s)}{L_{K^{G^i}}(s)}\]
is entire, and cannot have any simple zero. 
\end{theorem}
\begin{rem} If we assume Birch and Swinnerton-Dyer conjecture \ref{BSD}, then the result above says $\mbox{Rank}(E/_{K})-\mbox{Rank}(E/_{K^{G^i}})$ can not be one. Which is perhaps difficult to prove unconditionally, i.e, without assuming Birch and Swinnerton-Dyer conjecture. So there seems to have some scope of further research here.
\end{rem}
\section{Supercharacter theory on Artin L-functions}
Recently Diaconis and Isaacs \cite{DI} introduced the theory of supercharacters which generalizes the classical character theory in a natural way as follows.
\begin{definition} Let $G$ be a finite group, let $K$ be a partition of $G$, and let $X$ be a partition of $\emph{Irr}(G)$. The ordered pair $(X,K)$ is a super character theory if  $\{1\} \in K,|X|=|K|$ and for each $x \in X$, the character $\sigma_x =\sum_{\sigma \in x} \sigma(1)\sigma$ is constant on each $k \in K$. The characters $\sigma_x$ are called \textit{supercharacters}, and the elements $k$ in $K$ are called \textit{superclasses}. In addition, if $f : G \to C$ is constant on each superclass in $G$, then we say $f$ is a superclass function on $G$. 
\end{definition}
\noindent It is clear that the irreducible characters and conjugacy classes of $G$ give a supercharacter theory of $G$, which will be referred to as the \textit{classical theory} of $G$. Throughout this section, we will often equip groups with different kind of supercharacter theories to understand more about Artin L-functions. Diaconis and Isaacs \cite{DI} showed that every superclass is a union of conjugacy classes in $G$. Given a supercharacter theory on $G$, we denote the set of all supercharacters as $\mbox{Sup}(G)$. They also showed $\mbox{Sup}(G)$ forms an orthogonal basis for the inner product space of all superclass functions on $G$ with respect to the usual inner product on $\mbox{Irr}(G)$, 
\begin{definition} Let $G$ be a finite group and $H$ is a subgroup of $G$. We define $\mbox{SCl}_H(h)$ and $\mbox{SCl}_G(h)$ to be superclasses containing $h$ in $H$ and $G$, respectively. $G$ and $H$ are said to be compatible if for any $h \in H, \mbox{SCl}_H(h)$ is contained in $\mbox{SCl}_G(h)$
\end{definition}
\begin{definition}[Superinduction] Suppose $G$ and $H$ are compatible. Let $\phi$ be a supercalss functions on $H$. Then define superinduction of $\phi$ to $G$ by
\[S \emph{Ind}_H^G \phi(g)= \frac{|G|}{|H|}\frac{1}{|\mbox{SCl}_G(g)|}\sum_{i=1}^{m(g)}|\mbox{SCl}_{H}(x_{i,g})|\phi(x_{i,g}),\]
where $\mbox{SCl}_G(g)$ is the superclass in $G$ containing $g$ and $\{x_{i,g}\}$ is a set of superclass representatives in $H$ belonging to $\mbox{SCl}_G(g)$.
\end{definition}
One can check that $S\mbox{Ind}_H^G \phi(g)$ forms is a superclass function of $G$. 
\noindent Similarly analogous to Frobenius reciprocity, we have the following theorem in the case of our supercharacter theory. 
\begin{theorem}[Super Frobenius Reciprocity, \cite{Won}] Suppose that $G$ and $H$ are compatible. For all superclass functions $\phi$ on $H$ and all superclass functions $\psi$ of G, 
\[(S\emph{Ind}^G_H \phi, \psi) = (\phi, \psi|_H).\]
\end{theorem}
\begin{proof} Proof is no different from the usual proof of Frobenius reciprocity. See \cite{Won} pp. 21 for a complete proof. 
\end{proof} 
\begin{rem} Our defined superinduction is unique, (provided the superinduction has to satisfy super Frobenius reciprocity). Why ? Suppose that there is another map $\phi \mapsto \phi^{G}$, sending a superclass function on $H$ to a superclass function on $G$. Then, for any superclass function $\phi$ on $H$ and any superclass function $\Theta$ of $G$ we have,
\[(\phi^{G}, \Theta) = (\phi, \Theta|_{H}).\]
But right hand side above is same as $(S\mbox{Ind}^G_H \phi,\Theta)$. We know $\mbox{Sup}(G)$ forms an orthogonal basis for the inner product space of all supeclass functions on $G$, so varying $\Theta$ over $\mbox{Sup}(G)$ we get
$\phi^{G}=S\mbox{Ind}^{G}_H\phi$, as desired. 
\end{rem}  
\noindent Interesting part is via the theory of supercharacters and superinduction discussed above, one can generalize the classical Heilbronn character in the following way.
\begin{definition}[Super Heilbronn Characters] \label{Sind} Let $K/F$ be a Galois extension of number fields with Galois group $G$. Let $H$ be a subgroup of $G$. Let $G$ and $H$ be
compatible, $\mbox{Sup}(G)$ be the set of all supercharacters of $G$, and $\mbox{Sup}(H)$ be the set of all supercharacters of $H$. Assume that the restriction $\sigma|_{H}$ of any supercharacter $\sigma$ of $G$ to $H$ is an integral combination of supercharacters of $H$. Then the super Heilbronn
character $\Theta_H$ (with respect to $s = s_0$) is defined by
\[\Theta_{H} =\sum_{\tau \in \emph{Sup}(H)} n(H,\tau) \frac{\tau}{\tau(1)},\]
where $n(H, \tau ) = \frac{1}{m} \emph{ord}_{s=s_0} L(s,mS\mbox{Ind}^G_H\tau,K/K^{H})$, and $m = \emph{lcm}\Big(\sigma(1), \sigma \in \emph{Sup}(G)\Big).$
\end{definition} 
One might naturally ask why there are extra $m$ and $\frac{1}{m}$ for each $n(H, \tau )$. This is because, it is always natural to consider Artin L-functions attached to characters, and we can write
\[\mbox{ord}_{s=s_0} L(s, mS\mbox{Ind}^G_H \tau, K/F)=\mbox{ord}_{s=s_0}L\Big(s,\sum_{\sigma \in \mbox{Sup}(G)}\frac{m(S\mbox{Ind}^G_H \tau, \sigma)\sigma}{(\sigma,\sigma)},K/F\Big).\]
From the given condition, $(S\mbox{Ind}^G_H \tau, \sigma)$ is an integer. So to put the arguments in perspective, we declare $(\sigma,\sigma)$($=\sigma(1)$) to divide $m$ for all supercharacter $\sigma$ of $G$. So now we can consider Artin L-functions attached to supercharacters and study them. Note that, if one considers the improper subgroup $H$ of $G$, i.e., $H = G$, equipped with the same supercharacter theory, then the superinduction from $H$ to $G$ is the identity map, i.e, for any supercharacter $\sigma$ of $H = G,S\mbox{Ind}^G_H \sigma = \sigma$. So, 
\[n(G, \sigma) =\frac{1}{m} \mbox{ord}_{s=s_0} L(s,m\sigma,K/F) = \mbox{ord}_{s=s_0} L(s,\sigma,K/F),\]
which coincides with the classical definition. Point of doing all these is to understand how Artin L-functions attached to supercharacters satisfy similar properties with Artin L-functions attached to irreducible characters. Our main theme of this section is to understand Heilbronn characters in various settings. We already defined what super Heilbronn characters are. The definition at \ref{Sind} of course should coincide with original one when we we consider classical theory. Yes indeed it does, which only relies on the fact that any supercharacter in this case looks like $\chi(1)\chi$ such that $\chi \in \mbox{Irr}(G)$.
\begin{align}
\Theta_{G} & =\sum_{\sigma \in \mbox{Sup}(G)}n(G,\sigma)\frac{\sigma}{\sigma(1)}\\
& =  \sum_{\chi \in \mbox{Irr}(G)} n(G,\chi(1)\chi)\frac{\chi}{\chi(1)}\\
& = \sum_{\chi \in \mbox{Irr}(G)} n(G,\chi)\chi,
\end{align}
as required.\\ 
\newline
Next we present the following result that generalizes the previous works of Foote and Murty, in the context of supercharacters.
\begin{theorem}[Wong, \cite{Won}] \label{lo} Let $K/F$ be a Galois extension of number fields with Galois group $G$. One has
\[\sum_{\sigma \in G} \frac{n(G,\sigma)^2}{\sigma(1)} \leq (\emph{ord}_{s=s_0} \zeta_K(s))^2.\] 
\end{theorem}
\begin{proof} For every $\sigma \in \mbox{Sup}(G)$, one can write $\sigma$ as \[\sigma=\sum_{\chi \in \mbox{Irr}(G,\sigma)} \chi(1) \chi,\]
where $\mbox{Irr}(G, \sigma)$ is the set of irreducible characters of $G$ appearing in $\sigma$. Then,
\[\mbox{ord}_{s=s_0} L(s, \sigma,K/F) =\sum_{\chi \in \mbox{Irr}(G,\sigma)} \chi(1)\mbox{ord}_{s=s_0} L(s, \chi,K/F).\]
By Cauchy-Schwarz inequality we have,
\[n(G,\sigma)^2 \leq \Big(\sum_{\chi \in \mbox{Irr}(G,\sigma)} \chi^2(1)\Big) \Big(\sum_{\chi \in \mbox{Irr}(G,\sigma)} \big(\mbox{ord}_{s=s_0} L(s,\chi,K/F)\big)^2\Big).\]
Also by orthogonality property we have,
\[\sigma(1)=\sum_{\chi \in \mbox{Irr}(G,\sigma)} \chi^2(1).\]
Now the proof immediately follows from the fact that $\mbox{Irr}(G) = \bigsqcup_{\sigma \in \mbox{Sup}(G)}\mbox{Irr}(G, \sigma)$ and Foote-Murty's inequality \ref{MRI}.
\end{proof}
\noindent As was in the case for Heilbronn characters, we should also expect to have compatibility for super Heilbronn characters. The following lemma indeed fulfill our expectation.
\begin{lemma}[Super Heilbronn-Stark lemma, \cite{Won}] We have \[\Theta_G|_H=\Theta_H,\]
for any subgroup $H$ of $G.$
\end{lemma} 
\begin{proof} This comes from a straightforward application of super Frobenius reciprocity, as was the case for Heilbronn Stark lemma \ref{Stark} before. See page 58 of \cite{Won} for more detailed proof. 
\end{proof}
\noindent Let us now prove the theorem below which is an analogous of Stark's result \ref{Stark}.
\begin{theorem}[Weak super-Stark, \cite{Won}] 
If $\emph{ord}_{s=s_0} \zeta_K(s) = 0$, then Artin L-functions 
\[L(s,mS\emph{Ind}^G_H \chi,K/F)\] 
attached to supercharacters $\chi$ of $H$ are holomorphic and non-vanishing at $s = s_0$.
\end{theorem}

\noindent \textit{Proof.} The proof clearly follows from the following lemma, 
\begin{lemma} 
\[ \frac{|H|}{|G|} \sum_{\tau \in \emph{Sup}(H)} \frac{n(H,\tau)^2}{\tau(1)^2} \leq (\emph{ord}_{s=s_0} \zeta_K(s))^2.\]
\end{lemma}
\begin{proof} Note that
\begin{align}
|\Theta_{G}|^2 & = \sum \frac{n(G,\sigma)^2}{\sigma(1)}\\
& \leq (\mbox{ord}_{s=s_0} \zeta_K(s))^2,
\end{align}
where (5.7) comes from (5.6), by Theorem \ref{lo}. But on the other hand, 
\begin{align}
(\Theta_G,\Theta_G) & = \frac{1}{|G|}\sum_{g \in G} \Theta(g)\overline{\Theta(g)}\\
& \geq \frac{1}{|G|}\sum_{g \in H} \Theta(g)\overline{\Theta(g)}\\
& = \frac{|H|}{|G|} (\Theta_H,\Theta_H)\\
& =\frac{|H|}{|G|} \sum_{\tau \in \mbox{Sup}(H)} \frac{n(H,\tau)^2}{\tau(1)^2},
\end{align}
as desired. 
\end{proof}
\noindent Our purpose of this section is to study Artin L-functions attached to supercharacters. More specifically, we intend to investigate under which supercharacter theories on $G$, Artin's conjecture is true for all Artin $L$-functions associated to supercharacters. This is perhaps a very difficult topic, because we are still struggling with the classical case (when supercharacter theory is given by conjugacy classes and usual characters). So the first question one may ask, 
\begin{question} 
Does there exist any supercharacter theory for which Artin's conjecture is true ?  
\end{question} 
Fortunately we are lucky here by virtue of the Aramata-Brauer theorem \ref{AB1}, which says that, for any Galois extension $K/F$ with Galois group $G$, the quotient $\zeta_K(s)/\zeta_F(s)$ is entire. But $\zeta_K(s)/\zeta_F(s)$ is nothing but Artin L-function attached to the super character $\text{reg}_G-1_G$. So, Artin's conjecture holds true if we invoke the theory given by $\{\text{reg}_G-1_G,1_G\}$. We refer this as \textit{max theory}\label{MT}. In the light of this observation, we consider the following conjecture, which may be seen as a supercharacter-theoretic variant of Artin's conjecture.
\begin{conjecture} Let $K/F$ be a Galois extension of number fields with Galois group $G$. Let $(X,K)$ be a supercharacter theory of $G$ and $\mbox{Sup}(G)$ denote the set of supercharacters with respect to $(X,K)$. Then for every non-trivial $\sigma \in \mbox{Sup}(G)$, the
Artin L-function $L(s, \sigma, K/F)$ attached to $\sigma$ extends to an entire function. For such an instance, we shall say that Artin's conjecture holds for $G$ with respect to $(X,K)$.
\end{conjecture} 
Hendricson \cite{HEN} introduced the notion of product of supercharacter theories. We will describe below this notion and see how it helps to understand more of this vast topic.
\begin{definition} Let $G$ be a finite group and $N$ be a normal subgroup of $G$. We equip $N$ and $G/N$ with supercharacter theories $(C_1,P_1)$ and $(C_2,P_2)$ respectively. Further we assume $(C_1,P_1)$ is $G$-invariant in the sense, for each $g \in G$ and $n \in N$, both $n$ and $g^{-1}ng$ belong to the same superclass. Define their product $(C_1,P_1) \otimes (C_2,P_2)$ as the pair $(C,P)$ where  
\[C=\big\{\mbox{Ind}^G_N(\sigma_X) \mid X \in C_1-\{1_N\} \big\} \cup \big\{\mbox{Inf}_{G/N}^{G} \sigma_Y \mid Y \in C_2 \big\},\]
and  
\[P=P_1 \cup \{Np \mid p \in P_2-\{e_{G/N}\}\},\]
where $\mbox{Inf}$ is meant to be inflation of supercharacters from $G/N$ to $G$.
\end{definition} 
\noindent Hendrickson \cite{HEN} proved that the pair $(C,P)$ indeed gives a supercharacter theory of $G$. The main point behind introducing such product is to observe the following: If (super character theoretic) Artin's conjecture for both of the theories are true, then it is true for their product as well. This is because Artin L-functions are invariant under induction and inflation property (Lemma \ref{first}). At a first glance it may not sound very useful, but imagine if we know Artin's conjecture is true for classical theory on $G/N$, then it is of course true for any theory on $G/N$ (because any super character is positive integral combination of irreducible characters) and if we equip $N$ with max theory \ref{MT}, then for a whole lot of supercharacter theory on $G$, Artin's conjecture is true. For instance if $N=G^1=[G,G]$ then $G/N$ is abelian, and so Artin's conjecture is true for classical theory on $G/N$. \\

\section{Further Remarks}



In 2008, F. Nicolae \cite{FNIC} studied the Artin L-functions, which are holomorphic at a particular point $s=s_0$. We call this set to be $\mbox{Hol}(s_0)$. Nicolae showed that $\mbox{Hol}(s_0)$ is a finitely generated semigroup, using a result about positive polyhedral cones in the theory of convexity and optimization. Moreover Nicolae has shown that Artin's conjecture is true at $s=s_0$ if and only if, minimal number of generators of seimigroup is $|\mbox{Irr}(G)|$ and any $|\mbox{Irr}(G)|-1$ fold product of irreducible Artin L-functions (L-functions attached to irreducible characters) is holomorphic at $s_0$. This observation led him to prove the conjecture for icosahedral extension over $\mathbb{Q}$, using a result by Artin \cite{Artin} pp. 108. It is interesting to note that the idea of working with cones was earlier considered by S. Rhoades \cite{Rhoades} in $1993$. We give a brief sketch of Rhoades's work. 
\begin{definition} A complex valued function $\chi$ on $G$ is said to be a virtual character if $\chi=\sum_{\phi_i \in \emph{\emph{Irr}}(G)} r_i \phi_i$, where $r_i$'s all are in $\mathbb{R}$. Let $F$ be a subset of all characters of $G$. A \textit{virtual character} $\chi$ of $G$ is said to be a Heilbronn character with respect to $F$ if $(\chi,\psi) \geq 0$ for all $\psi \in F$. \end{definition} 
\begin{theorem}\label{RD} Let $F$ be a nonempty set of characters of the finite group $G$, and let $\psi$ be a nonzero virtual character of $G$. Then $\psi$ can be written as a positive rational linear combination of characters from $F$ if and only if $\psi$ is a Heilbronn character with respect to $F$.
\end{theorem}
\begin{proof} Let $k=\mid \mbox{Irr}(G)\mid$ and $\langle,\rangle$ denote the usual inner product on $\mathbb{R}^k$. Let $F$ be any nonempty subset of $\mathbb{R}^k$. Define 
\[S(F)=\Big\{x\in \mathbb{R}^k \mid \langle f,x\rangle \geq 0, \forall f \in F\Big\},\]
and 
\[T(F)=\Big\{\sum w_if_i \mid f_i \in F, w_i \in \mathbb{R}^{+}\Big\}.\]
\newline
\noindent\textbf{Step 1:} If $F$ is non empty not containing $0$, then $S(S(F))=T(F)$.\\
\newline
\noindent\textbf{Step 2:} Let $v_1, \cdots, v_s$ and $V$ be vectors in $\mathbb{R}^k$ with nonnegative rational coordinates. If $V$ can be written as a nonnegative real linear combination of $v_1,\cdots,v_s$, then $V$ can also be written as a nonnegative rational linear combination of $v_1, \cdots, v_s$.\\
\newline
\noindent\textbf{Final Step:} Identifying $\psi$ with $\mathbb{R}^k$ we have $\psi \in S(S(F))$.
\end{proof} 
As a byproduct, Rhoades obtained a different proof of Brauer-Aramata theorem (Theorem \ref{BAT}). In fact by Rhode's result \ref{RD}, we can say that $\chi$ is an irreducible character of $G$ such that $(\chi, \psi)\geq 0$ for all monomial characters $\psi$ of G, then $\chi$ is a positive rational linear combination of monomial characters. And this allows us to write $k\chi$ as positive integral combination of monomial characters, for some natural number $k$. And this in turn says that $L(s,\chi,K/F)^k$ is entire. This indeed makes $L(s,\chi,K/F)$ to be entire, because it is already known to meromorphic due to Theorem \ref{mero}. So could this be one approach towards making in roads to Artin's conjecture ?\\
\newline
Note that for the case of solvable groups, by the result of Ferguson and Isaacs \cite{FI} we can conclude that, if $G$ is solvable and any positive integer multiple of an irreducible character $\chi$ is monomial, then $\chi$ itself is monomial. However for monomial characters we know Artin's holomorphy conjecture to hold, from the day Artin L-functions were discovered in $1923$. The main problem is to see if this approach can shed some light for non-solvable groups.\\ 
\newline
It is interesting to note that recently Konig \cite{Kon} showed that, if for every $\chi \in \mbox{Irr}(G)$ there exist a natural number $k_{\chi}$ such that $k_{\chi}\chi$ is monomial, then $G$ is solvable ! To show solvability Konig embedded $G$ inside $\mbox{Aut}(S^m)$, where $S^m$ is a non-abelian simple group. In $S^m$ with $S$ nonabelian, the minimal normal subgroups are precisely the $m$ factors, which are thus permuted by the automorphism group. This gives the description of the automorphism group. To be precise $\mbox{Aut}(S^m)= \mbox{Aut}(s) \rtimes_{\rho} \mbox{Sym}(m)$, where $\rho: \mbox{Sym}(m) \longrightarrow \mbox{Aut}(\mbox{Aut}(S))$ is given by $\rho(\sigma)(f)(g)=f(g^{\sigma})$. After this the general problem was reduced to the problem of investigation of finite simple groups. After this Konig attacked the problem case by case relying on classification of finite simple groups. For a detailed arguments, see \cite{Kon}. This story tells us that we must look into other approaches for making any meaningful contribution further in this area of research. We hope to pursue these problems in near future.

\section{Acknowledgement}
I express my sincere gratitude to Professor K. Srinivas for his help in the making of note. His careful reading of earlier versions of this draft, suggestions to improve the presentation at various places is a matter of great education for me. I am especially grateful to Professor Anne-Marie Aubert for carefully reading the draft, suggesting changes which improved the content and presentation of this note. In all my needs, I always found her as a kind and helpful person. I am indebted to her for all the care and help that I received for the past one year. I also would like to thank CMI and IMSc for providing wonderful working environment. Last but not the least, I acknowledge the support and affection of my parents, cousins and Sulakhana Chowdhury who always stood by me in every choice of my life.

\end{document}